\documentclass[10pt,a4paper]{article} 

\usepackage{color}
\usepackage{a4,amsmath,amssymb,amscd,latexsym, amsthm} 
\usepackage{bbm}
\usepackage{epsfig,color} 
\usepackage{amsfonts}
\usepackage{mathrsfs}
\usepackage{overpic}
\usepackage{subcaption}
\usepackage{paralist}
\usepackage{graphicx}
\usepackage{trfsigns}
\usepackage{url}
\usepackage{stmaryrd}
\usepackage{booktabs}
\usepackage{algorithm, algorithmic}
\usepackage[algo2e]{algorithm2e} 
\usepackage{hyperref}
\usepackage{cleveref}

\usepackage{tikz, pgfplots}
\usetikzlibrary{calc,trees,positioning,arrows,chains,shapes.geometric,%
	decorations.pathreplacing,decorations.pathmorphing,shapes,%
	matrix,shapes.symbols,backgrounds,shadows}

\usepackage{tikz-cd}

\newtheorem{definition}{Definition}
\newtheorem{theorem}{Theorem}
\newtheorem{assumption}{Assumption}
\newtheorem{remark}{Remark}
\newtheorem{proposition}{Proposition}

\newenvironment{@abssec}[1]{%
	\vspace{.06in}\footnotesize
	\parindent=0in
	\ignorespaces 
	}

\newenvironment{keywords}{\begin{@abssec}{\keywordsname}}{\end{@abssec}}
\newenvironment{AMS}{\begin{@abssec}{AMS subject classification:}}{\end{@abssec}}

\title{\bf Efficient Techniques for Shape Optimization with Variational Inequalities using Adjoints}

\author{Daniel Luft\thanks{Trier University, Department of Mathematics, 54286 Trier, Germany (\tt luft@uni-trier.de)} \and Volker Schulz\thanks{Trier University, Department of Mathematics, 54286 Trier, Germany (\tt volker.schulz@uni-trier.de)} \and Kathrin Welker\thanks{Helmut Schmidt University / University of the Federal Armed Forces, Faculty of Mechanical Engineering, 22043 Hamburg, Germany (\tt welker@hsu-hh.de)}} 

\newcommand{\R}{{\mathbb{R}}} 
\newcommand{\Gi}{\Gamma_\text{int}}

\begin{document}

	\maketitle
	
	\begin{abstract}
		\noindent
		In general, standard necessary optimality conditions cannot be formulated in a straightforward manner for semi-smooth shape optimization problems. 
		In this paper, we consider shape optimization problems constrained by variational inequalities of the first kind, so-called obstacle-type problems.
		Under appropriate assumptions, we prove existence of adjoints for regularized problems and convergence to adjoints of the unregularized problem.
		Moreover, we derive shape derivatives for the regularized problem and prove convergence to a limit object. Based on this analysis, an efficient optimization algorithm is devised and tested numerically.
	\end{abstract}
	
	\begin{keywords}
		\textbf{Key words: } Semi-smooth optimization, variational inequality, obstacle problem, shape optimization, numerical methods, adjoint methods.
	\end{keywords}
	
	\begin{AMS}
		\textbf{AMS subject classifications: }
		65K15, 49Q10, 49M29, 35Q93, 35J86, 49J40. 
	\end{AMS}

	\section{Introduction}
	\label{section_model}
	
	We consider shape optimization problems constrained by variational inequalities (VI) of the first kind, so-called obstacle-type problems. Applications are manifold and arise, whenever a shape is to be constructed in a way not to violate constraints for the state solutions of partial differential equation depending on a geometry to be optimized. Just think of a heat equation depending on a shape, where the temperature is not allowed to surpass a certain threshold. This example is basically the model problem that we are formulating in \cref{section_ModelProblem}. 
	Applications of general VI's include contact problems in solid state mechanics, viscoplasticity and network equilibrium problems, and thus a wide range of industrial problems (cf. \cite{kikuchi1988contact,capatina2014variational, cocou2002existence,giannessi1995variational}).
	
	Shape optimization problem constraints in the form of VIs are challenging, since classical constraint qualifications for deriving Lagrange multipliers generically fail. Therefore, not only the development of stable numerical solution schemes but also the development of suitable first order optimality conditions is an issue.
	
	By usage of tools of modern analysis, such as monotone operators in Banach spaces, significant results on properties of solution operators of variational inequalities have been achieved since the 1960s (cf.~\cite{Br-1971,BrSt-1968,LiSt-1967}).
	However, there are only very few approaches in literature to the problem class of VI constrained shape optimization problems so far. In \cite{KO-1994}, shape optimization of 2D elasto-plastic bodies is studied, where the shape is simplified to a graph such that one dimension can be written as a function of the other.  
	The non-trivial existence of solutions of VI constrained shape optimization problems is discussed in \cite{Delfour_Zolesio,SokoZol}. 
	E.g., in \cite[Chap.~4]{SokoZol}, shape derivatives of elliptic variational inequality problems are presented in the form of solutions to again variational inequalities.
	In \cite{Myslinski-2001}, shape optimization for 2D graph-like domains are investigated. Also \cite{LR-1991a, LR-1991b} present existence results for shape optimization problems which can be reformulated as optimal control problems, whereas \cite{DM-1998, G-2001} show existence of solutions in a more general set-up. In \cite{Myslinski-2004, Myslinski-2007}, level-set methods are proposed and applied to graph-like two-dimensional problems. Moreover, \cite{HL-2011} presents a regularization approach to the computation of shape and topological derivatives in the context of elliptic variational inequalities and, thus, circumventing the numerical problems in \cite[Chap.~4]{SokoZol}.
	Recently, in \cite{Heinemann-Sturm-2016}, a sensitivity analysis is performed for a class of semi-linear variational inequalities and a strong convergence property is shown for the material derivative. Furthermore, state-shape derivatives are established under regularity assumptions. 
	
	In this paper, we aim at optimality conditions for VI constrained shape optimization in the flavor of optimality conditions for VI constrained optimal control problems as in \cite{hintermuller2008active,hintermuller2009mathematical,Hint-Surowiec-2011}. In general, standard necessary optimality conditions cannot be formulated in a straightforward manner for semi-smooth shape optimization problems. 
	Under appropriate assumptions, we prove existence of adjoints and convergence of adjoints resulting from regularized variational inequalities.
	These analytical results are also verified numerically.
	Moreover, convergence of shape derivatives related to the smoothed problem is shown and the limit object is identified. Furthermore, we build on the resulting optimality conditions and devise an optimization algorithm giving specific numerical results. This algorithm does no longer depend on smoothing strategies as in \cite{mPDAS}. In \cite{mPDAS}, a shape optimization method based on a regularized variant of the variational inequality has been devised and observed that the performance of this algorithm strongly depends on the tightness of the obstacle. This problem does no longer arise with the strategy developed in the present paper. On the contrary, the algorithms gets even faster, the more degrees of freedom are constrained by the obstacle.

	This paper is structured as follows. In \cref{section_ModelProblem}, we formulate the VI constrained shape optimization model with general elliptic coefficients on which we focus in this paper. 
	The necessary optimality conditions, including the existence of adjoint variables under certain regularity assumptions to the model problem are formulated in  \cref{section_Analytical}.
	In \cref{section_Algorithmic}, we formulate an algorithm to solve the model problem based on these analytical results and compare numerically this approach with several regularized strategies.

	\section{Problem class}
	\label{section_ModelProblem}
	
	Let $\Omega\subset \mathbb{R}^n$ be a bounded domain equipped with a sufficiently smooth boundary $\partial\Omega$, where $n\in\mathbb{N}$ is the dimension. For typical applications $n=2$ or $n=3$. This domain is assumed to be partitioned in a subdomain $\Omega_\text{out}\subset\Omega$ and an interior domain $\Omega_\text{int}\subset \Omega$ with boundary $\Gamma_\text{int}:=\partial\Omega_\text{int}$ such that $\Omega_\text{out}\sqcup \Omega_\text{int} \sqcup \Gamma_\text{int}=\Omega$, where $\sqcup$ denotes the disjoint union. The closure of $\Omega$ is denoted by $\bar{\Omega}$. We consider $\Omega$ depending on $\Gi$, i.e., $\Omega=\Omega(\Gi)$. Figure \ref{example-domain} illustrates this situation.
	In the following, the boundary $\Gamma_\text{int}$ of the interior domain is called the \emph{interface} and an element of an appropriate shape space $\mathcal{X}$ (cf.~\cref{ShapeSpace}). 
	In contrast to the outer boundary $\partial\Omega$, which is assumed to be fixed, the inner boundary $\Gi$ is variable. 
	If $\Gi$ changes, then the subdomains $\Omega_\text{int},\Omega_\text{out}\subset \Omega$ change in a natural manner.

	Let $\nu>0$ be an arbitrary constant. For the objective function
	\begin{equation}\label{objective}
	J(y,\Omega):=\mathcal{J}(y,\Omega)+ \mathcal{J}_\text{reg}(\Omega):=\frac{1}{2}\int_{\Omega} \left|y - \bar{y}\right|^2\; dx+\nu \int_{\Gi} 1\; ds	
	\end{equation}
	we consider the following shape optimization problem: 
	\begin{align}
	\min\limits_{\Gi\in\mathcal{X}}\; J(y,\Omega)\label{eq_minimization}
	\end{align}
	constrained by the following obstacle type variational inequality:
	\begin{align}
	\label{VI_general}
	a(y,v-y)\geq \left<f,v-y\right> \quad\forall v\in K:=\{\theta \in H^1_0(\Omega)\colon \theta(x)\leq \varphi(x) \text{ in }\Omega\},
	\end{align}
	where $y \in K$ is the solution of the VI, $f\in L^2(\Omega)$ is explicitly dependent on the shape, $\left<\cdot,\cdot\right>$ denotes the duality pairing and $a(\cdot,\cdot)$ is a general strongly elliptic, i.e. coercive, symmetric bilinear form 
	\begin{align}\label{bilinearform}
	\begin{split}
	a\colon H_0^1(\Omega)\times H_0^1(\Omega) &\rightarrow \mathbb{R} \\
	(y, v) &\mapsto \int_{\Omega}\underset{i,j}{\sum}   a_{i,j}\partial_iy  \partial_jv + \underset{i}{\sum}d_i( \partial_iy v + y \partial_iv) +  byv \; dx
	\end{split}
	\end{align} 
	defined by coefficient functions $a_{i,j}, d_j, b\in L^\infty(\Omega)$, fulfilling the weak maximum principle. However, the results of this paper still remain correct if symmetry of $a(.,.)$ is dropped as an assumption by simple modifications of proofs.
	
	
	With the tracking-type objective $\mathcal{J}$ the model is fitted to data measurements $\bar{y}\in H^1(\Omega)$.
	The second term $\mathcal{J}_\text{reg}$ in the objective function $J$ is a perimeter regularization. 
	A perimeter regularization is frequently used to overcome ill-posedness of inverse problems, e.g., \cite{Burger-2004} investigates the regularization  and  numerical  solution of  geometric  inverse problems  related  to  linear  elasticity.
	In \cref{VI_general}, $\varphi$ denotes an obstacle which needs to be an element of $L^1_{\text{loc}}(\Omega)$ such that 
	the set of admissible functions $K$ is non-empty (cf.~\cite{SokoZol}).
	If additionally $\partial\Omega$ is Lipschitzian and $\varphi \in H^1(\Omega)$ with $\varphi_{\vert\partial\Omega} \geq 0$, then there is a unique solution to \cref{VI_general} satisfying $y\in H_0^1(\Omega)$, given that the assumptions from above hold  (cf.~\cite{ito2000optimal,kinderlehrer1980introduction,troianiello2013elliptic}).
	Further, \cref{VI_general} can be equivalently expressed as
	\begin{align}\label{PDE}
	a(y,v)+(\lambda,v)_{L^2(\Omega)} &= (f,v)_{L^2(\Omega)} \quad \forall v\in H_0^1( \Omega)
	\end{align}
	\begin{align}\label{VI_conditions}
	\begin{split}
	\lambda &\geq 0 \quad \text{in } \Omega \\
	y  &\leq \varphi \quad \text{in } \Omega \\
	\lambda(y-\varphi) &= 0 \quad \text{in } \Omega 
	\end{split}
	\end{align}
	with $(\cdot,\cdot)_{L^2(\Omega)}$ denoting the $L^2$-scalar product and $\lambda\in L^2(\Omega)$.
	
	It is well-known, e.g., from \cite{kinderlehrer1980introduction}, that under these assumptions there exists a unique solution $y$ to the obstacle type variational inequality (\ref{VI_general}) and an associated Lagrange multiplier $\lambda$. 
	The existence of solutions of any shape optimization problem is a non-trivial question. Shape optimization problems constrained by VIs are especially challenging because, in general, it is not guaranteed that an adjoint state can be introduced (cf.~\cite[Example in Chap.~1, Chap.~4]{SokoZol}). 
	An essential theoretical tool for the study of the existence of solutions is the derivation of optimality conditions, i.e., in particular, the formulation of an adjoint equation. 
	Therefore, \cref{section_Analytical} investigates the model problem analytically, also in view of formulating a numerically applicable algorithm in \cref{section_Algorithmic}.
	
	\begin{figure}
		\vspace*{.6cm}
		\begin{center}
			\includegraphics[width=.35\textwidth]{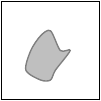}
			\put(-85,40){$\Omega_\text{int}$}
			\put(-30,13){$\Omega_\text{out}$}
			\put(-55,74){$\color{darkgray}\Gi $}
		\end{center}
		\caption[Example of a domain $\Omega$ and $X$]{Example of a domain $\Omega=\Omega_\text{out} \sqcup\Gi \sqcup\Omega_\text{int}$.}
		\label{example-domain}
	\end{figure} 
	
	\begin{remark}
		\label{ShapeSpace}
		The interface $\Gi$ is an element of an appropriate shape space. Please note that there exists no common shape space suitable for all applications. It should be mentioned, that the existence of shape derivatives and their form is not dependent on the explicit choice of a shape space, hence only requirements noted in the according theorems are necessary.
		From a computational point of view one has to deal with polygonal shape representations  arising in the setting of constrained shape optimization. This is owed to the fact that finite element methods usually discretize the models. 
		In this paper, we use Steklov-Poincar\'e metric as introduced in \cite{schulz2015Steklov}. These metrics can be considered e.g. on the space $B_e$ (cf. \cite{bauer2014overview}), or more generally on the space of $H^{1/2}$-shapes (cf. \cite{Welker_Diffeological}). In \cite{SiebenbornWelker_skin}, it is outlined that this is an essential step towards applying efficient FE solvers. Of course, it is possible to choose other shape space models, but this is beyond the scope of this paper.
	\end{remark}

	\section{Convergence results for adjoints and shape derivatives}
	\label{section_Analytical} 
	We assume the situation mentioned in \cref{section_ModelProblem}, which is also found in \cite{ItoKunisch_VI}, giving us $\lambda\in L^2(\Omega)$. It can be easily verified that this in turn gives the possibility to summarize the conditions (\ref{VI_conditions}) equivalently into a single condition of the form 
	\begin{equation}
	\label{lambda}
	\lambda=\max\big(0,\lambda+c(y-\varphi) \big)\quad \text{for any }c>0.
	\end{equation}
	The direct handling of general obstacle-type variational inequalities formulated as in (\ref{PDE})-(\ref{VI_conditions}), with \cref{VI_conditions} being equivalently substitutes by \cref{lambda}, poses several challenges.
	One challenge of the solution of \cref{PDE} is the occurrence of distributional numerical iterates for $\lambda$ in $H^{-1}(\Omega)$ when an augmented Lagrangian approach is applied to \cref{PDE} constrained by \cref{lambda}, despite the analytical solution $\lambda$ having $L^2(\Omega)$-regularity. For a more detailed discussion of this, see \cite[p.2]{ItoKunisch_VI}.
		In order to circumvent the occurrence of distributions in the solution scheme, the authors of \cite{ItoKunisch_VI} introduce a relaxation for relation \cref{lambda} with a given regularization parameter $\alpha\in(0,1)$
	\begin{equation}
	\lambda=\alpha\cdot\max\big(0,\lambda+\varrho(y-\varphi) \big)\quad \text{for any }\varrho>0
	\end{equation} 
	which in turn is equivalent to
	\begin{equation}
	\lambda=\max\big(0,\bar{\lambda}+c(y-\varphi) \big) \quad c \in (0,\infty)
	\end{equation}
	if $\bar{\lambda} = 0$ and $c = \frac{\varrho\alpha}{1-\alpha}$, where $\bar{\lambda}\in L^2(\Omega)$ can be motivated by updates of the augmented Lagrangian.
	This results in the equation
	\begin{equation}\label{RegularizedState}
	a(y_c, v) + (\max\big(0, \bar{\lambda} + c(y_c - \varphi)\big), v)_{L^2(\Omega)} = (f,v)_{L^2(\Omega)} \qquad \forall v \in H^1_0(\Omega),
	\end{equation}
	which in the following is called \emph{regularized state equation} or \emph{relaxed obstacle problem}. Explicit dependence on $\lambda$ is avoided, making the resulting semi-linear elliptic equation tractable, for example by semi-smooth Newton methods, see, e.g, \cite{ItoKunisch_VI}. 
	Moreover, the authors of \cite{ItoKunisch_VI} prove $L^2$-convergence of the regularized multiplier $\max\big(0, \bar{\lambda} + c\cdot(y_c - \varphi)\big)$ to the original $\lambda$ for their proposed semi-smooth Newton method.
	
	With problem (\ref{RegularizedState}) we are still left to solve a nonlinear, semi-smooth problem, giving rise to problems concerning existence of adjoints for the shape optimization problem. Hence, standard smoothing strategies can be applied to render this problem smooth enough to show existence of adjoints and to apply techniques such as Newton iterations. 
	
	In light of \cite{schiela2013convergence} and \cite{christof2017optimal}, we pose the following assumptions on the smoothed $\max$-function, which from now on is called $\max_{\gamma}\colon \mathbb{R} \rightarrow [0, \infty)$, with $\gamma > 0$ being the smoothing parameter:
	
	\begin{assumption}[on smoothed $\max$-function]\
		\begin{enumerate}
			\label{AssumptionsOnMax}
			\item[(i)] $\max_{\gamma}\in C^1(\Omega)$ for all $\gamma > 0$;
			\item[(ii)] there exists a function $g\colon (0,\infty) \rightarrow [0,\infty)$ with $g(\gamma) \rightarrow 0$ as $\gamma \rightarrow \infty$, s.t. $\vert\max_{\gamma}(x) - \max(0,x)\vert \leq g(\gamma)$ for all $x\in \mathbb{R}$ and for all $\gamma >0$;
			\item[(iii)] $\max_\gamma'(x) \in [0,1]$  and monotonically nondecreasing for all $x\in \mathbb{R}$ and all $\gamma >0$;
			\item[(iv)] $\max_\gamma'$ converges uniformly to $0$ on $(-\infty, -\delta)$ and $1$ on $(\delta, \infty)$ for all $\delta >0$ for $\gamma \rightarrow \infty$.
		\end{enumerate}
	\end{assumption}
	
	In the following, let $\text{sign}_\gamma$ denote the derivative of $\max_\gamma$.
	An example satisfying these assumptions is given in (\ref{SmoothingMaxExample}).
	Applying $\max_\gamma$ instead of $\max$ in (\ref{RegularizedState}) gives the following equation, which we call \emph{fully regularized state equation} in the subsequent chapters:
	\begin{equation}\label{SmoothedState}
	a(y_{\gamma, c},v)+\left(\text{max}_\gamma\big(\bar{\lambda} +  c(y_{\gamma, c}-\varphi)\big),v\right)_{L^2(\Omega)}=(f,v)_{L^2(\Omega)}\quad\forall v\in H_0^1(\Omega).
	\end{equation}
	
	So linearizing the corresponding Lagrangian with respect to $y_{\gamma,c}$ results in the typical adjoint equation
	\begin{equation}\label{SmoothedAdjoint}
	\begin{split}
	&	a(p_{\gamma, c}, v) + c\cdot\left(\textup{sign}_\gamma(\bar{\lambda} +  c(y_{\gamma, c}-\varphi))\cdot p_{\gamma,c},v\right)_{L^2(\Omega)}\\
	& =-(y_{\gamma,c} - \bar{y},v)_{L^2(\Omega)}\quad\forall v\in H_0^1(\Omega)
	\end{split}
	\end{equation}
	(see, e.g., \cite{hintermuller2008active} or \cite{schiela2013convergence} in the context of optimal control).
	
	\begin{remark}
		As in \cite{schiela2013convergence}, smoothness of the state equation (\ref{SmoothedState}) in $y_{\gamma,c}$ guarantees existence of solutions to the linearized equation (\ref{SmoothedAdjoint}) for a given $L^2(\Omega)$ right-hand side and, thus, existence of adjoints in the case of the considered tracking-type objective functional (\ref{objective}).
	\end{remark}
	
	\subsection{State and adjoint equation}
	\label{subsection_StateAdjoint}
	
	We first show that solutions of \cref{SmoothedState} converge strongly in $H^1$ to solutions of (\ref{PDE})-(\ref{VI_conditions}) for $\gamma,c \rightarrow \infty$. This is proven in \cite{schiela2013convergence} for stronger assumptions on the smoothed function $\max_\gamma$ and under $\gamma = c$ . 
	Since we rely on the general case $\gamma \neq c$ for the proofs in ongoing discussions, we state an according result.
	The first part of the following theorem is in analogy to \cite[Lemma~4.2]{christof2017optimal}. However, the difference is that we consider general elliptic bilinear forms and---more importantly---a modified argument in the maximum function resulting in different regularized state equations. These generalizations are necessary for our further analytical investigations leading to an adjoint equation.
	\begin{proposition}[$H^1$-convergence of the state]\label{Theorem0}
		Let $y_{\gamma, c}$, $y_c$ and $y$ be solutions to \cref{SmoothedState}, \cref{RegularizedState} and \cref{PDE}, respectively. Here, $a(\cdot,\cdot)$ is chosen by an elliptic bilinear form as in (\ref{bilinearform}) on a bounded, open domain $\Omega\subset \R^n$ with Lipschitz boundary, $f\in L^2(\Omega)$ and $\gamma, c >0$. Moreover, assume $\varphi \in H^1(\Omega)$, $\bar{\lambda}\in L^2(\Omega)$ and let $\max_\gamma\colon \mathbb{R} \rightarrow \mathbb{R}$ satisfy \cref{AssumptionsOnMax}. 
		
		Then \cref{SmoothedState} and \cref{RegularizedState} posses unique solutions and 
		\begin{align}\label{convergenceStateSmoothed}
		y_{\gamma,c}& \to y_c \text{ in }H^1(\Omega) \quad \text{ as }\gamma\to\infty;\\
		\label{convergenceState}
		y_{c}& \to y \,\text{ in }H^1(\Omega)\quad \text{ as }c\to\infty.
		\end{align} 
	\end{proposition}
	
	\begin{proof}
		We prove statement (\ref{convergenceStateSmoothed}) of the theorem. For a proof of statement (\ref{convergenceState}), we refer to  \cite[Theorem 3.1]{ItoKunisch_VI}.
		
		We start by ensuring the existence of solutions to \cref{SmoothedState} and \cref{RegularizedState}. For this, we show that the Nemetskii-operator defined by 
		\begin{equation} \label{NymetskiyGamma}
		\Phi_\gamma\colon H^1(\Omega) \rightarrow L^2(\Omega), 
		y \mapsto \text{max}_\gamma\big(\bar{\lambda} + c\cdot(y-\varphi)\big)
		\end{equation}
		is a monotone operator for all $\gamma, c>0$. Due to \cref{AssumptionsOnMax}, it is clear that $\text{max}_\gamma\colon \mathbb{R} \rightarrow \mathbb{R}$ is a point-wise monotone function, implying that $\text{max}_\gamma\colon H^1(\Omega) \rightarrow H^1(\Omega), y \mapsto \text{max}_\gamma(y)$ is a monotone operator. Since $$\Psi_c\colon H^1(\Omega) \rightarrow H^1(\Omega), \; y \mapsto \bar{\lambda} + c\cdot(y - \varphi)$$ is an affine linear operator, and, thus monotone, the composition $\text{max}_\gamma \circ \Psi_c = \Phi_\gamma$ is also monotone.
		The same argument holds for the non-smoothed operator 
		\begin{equation*}
		\Phi\colon H^1(\Omega) \rightarrow L^2(\Omega), 
		y \mapsto \text{max}\big(0,\bar{\lambda} +  c\cdot(y-\varphi)\big).
		\end{equation*}
		Therefore, applying the Browder-Minty theorem for monotone operators yields the existence of unique solutions to \cref{SmoothedState} and \cref{RegularizedState} in $H^1(\Omega)$ for all $f\in L^2(\Omega)$ if $\Omega$ is bounded and we operate in Hilbert spaces. \\
		Now, we prove the second convergence (\ref{convergenceStateSmoothed}). For fixed $c>0$, let $y_{\gamma, c}$ and $y_c$ be solutions to \cref{SmoothedState} and \cref{RegularizedState}, respectively. Assumption~\ref{AssumptionsOnMax} (ii) together with the monotonicity of $\Phi$, the coercivity of $a(\cdot,\cdot)$ with constant $K >0$ and $y_{\gamma, c} - y_c \in H^1(\Omega)$ acting as a test-function yields
		\begin{align*}
		0 &\leq K \cdot\vert\vert y_{\gamma, c} - y_c \vert\vert_{H^1(\Omega)}^2 \\
		&\leq a( y_{\gamma, c} - y_c,  y_{\gamma, c} - y_c) \\
		&\leq a( y_{\gamma, c} - y_c,  y_{\gamma, c} - y_c) \\
		&\hspace{.4cm}+ \Big(\text{max}\big(0, \bar{\lambda} + c(y_{\gamma, c} - \varphi)\big) - \text{max}\big(0, \bar{\lambda} + c(y_c - \varphi)\big),  y_{\gamma, c} - y_c\Big)_{L^2(\Omega)} \\
		&= \Big(\text{max}_\gamma\big(\bar{\lambda} + c(y_{\gamma, c} - \varphi)\big) - \text{max}\big(0, \bar{\lambda} + c(y_{\gamma, c} - \varphi)\big),  y_{\gamma, c} - y_c\Big)_{L^2(\Omega)} \\
		&\leq \int_{\Omega} \vert \text{max}_\gamma\big( \bar{\lambda} + c(y_{\gamma, c} - \varphi)\big) - \text{max}\big(0,\bar{\lambda} + c(y_{\gamma, c} - \varphi)\big)\vert \cdot \vert y_{\gamma, c} - y_c \vert \;dx \\
		&\leq g(\gamma)\cdot\text{vol}(\Omega)^{\frac{1}{2}}\cdot\vert\vert y_{\gamma, c} - y_c \vert\vert_{H^1(\Omega)},
		\end{align*}
		which gives the desired convergence (\ref{convergenceStateSmoothed}).
	\end{proof}
	
	The following definition is needed to state the first main result of this paper, the convergence of adjoints. 
	\begin{definition}\label{RegDecompoDefi}
		Let $\Omega \subset \mathbb{R}^n$ be a bounded, open domain with Lipschitz boundary. A set $A \subseteq \Omega$ is called regularly decomposable, if there exists an $N \in \mathbb{N}$ and path-connected, bounded and open $A_i \subset \Omega$ with Lipschitz boundaries $\partial A_i$ such that $A = \overset{N}{\underset{i=1}{\sqcup}} \bar{A_i} $ is a disjoint union.
	\end{definition}
	
	With this definition it is possible to formulate the first main theorem concerning the convergence of adjoints corresponding to the fully regularized problems and characterization of the limit object.
	
	%
	
	\begin{theorem}[Convergence of the adjoints]
		\label{Main_theorem_1}
		Let $\Omega \subset \mathbb{R}^n$ for $n\leq 4$ be a bounded, open domain with Lipschitz boundary. Moreover, let the following assumptions are satisfied:
		\begin{itemize}
			\item[(i)] $\varphi \in H^1(\Omega)$, $f\in L^2(\Omega)$, $\bar{y}\in H^1(\Omega)$ and coefficient functions $a_{i,j}, d_j, b\in L^\infty(\Omega)$ in (\ref{PDE})-(\ref{VI_conditions});
			\item[(ii)] the active set $A = \{ x\in \Omega \;\vert\; y - \varphi  \geq 0 \}$ corresponding to (\ref{PDE})-(\ref{VI_conditions}) is regularly decomposable;
			\item[(iii)]  $A_c :=  \{ x \in \Omega \;\vert\; \bar{\lambda} + c\cdot(y_c - \varphi)  \geq 0 \}$ is regularly decomposable and
			\begin{equation}\label{ActiveSetAssump}
			A_c \subseteq A \quad \forall c>0,
			\end{equation}
			where $y_c$ solves the regularized state equation (\ref{RegularizedState});
			\item[(iv)] the following convergence holds:
			\begin{equation}\label{AssumptionOnSign}
			\Vert\textup{sign}_\gamma(\bar{\lambda} + c\cdot(y_{\gamma, c} - \varphi)) - \textup{sign}(\bar{\lambda} + c\cdot(y_{c} - \varphi))\Vert_{L^1(\Omega)} \rightarrow 0 \quad \text{for } \gamma \rightarrow \infty. 
			\end{equation}
		\end{itemize}
		
		Then the adjoints $p_{\gamma, c} \rightarrow p_c$ in $H^1_0(\Omega)$ for $\gamma \rightarrow \infty$ for all $c>0$, where $p_c$ is the solution to
		\begin{align}\label{adjoint_limitproblem}
		\begin{split}
		a(p_c, v) + c \cdot\int_{\Omega}\mathbbm{1}_{A_c} \cdot  p_c \cdot v \; dx &= -\int_{\Omega}(y_c - \bar{y})\cdot v\; dx \quad \forall v\in H^1_0(\Omega).
		\end{split}	
		\end{align}
		Moreover, there exists $p\in H^{-1}(\Omega)$ to (\ref{PDE})-(\ref{VI_conditions}) and $p$ is representable as an $H^1_0$-function given by the extension of $\tilde{p} \in H^1_0(\Omega\setminus A)$ to $\bar{\Omega}$, i.e.,
		\begin{align}\label{adjoint_unregu}
		p = 
		\begin{cases}
		\tilde{p} \quad \text{ in } \Omega\setminus A \\
		0 \quad \text{ in } A
		\end{cases},
		\end{align}
		where $\tilde{p}\in H^1_0(\Omega\setminus A)$ is the solution of the elliptic problem 
		\begin{align}
		\begin{split} \label{adjoint_subproblem}
		\: \quad a_{\Omega\setminus A}(\tilde{p}, v) &= -\int_{\Omega\setminus A}(y-\bar{y})v\; dx \quad  \quad \forall v\in H^1_0(\Omega\setminus A)
		\end{split}
		\end{align}
		with 
		\begin{align}\label{BilinRestricted}
		\begin{split}
		a_{\Omega\setminus A}:  \quad &H^1_0(\Omega\setminus A) \times H^1_0(\Omega\setminus A) \rightarrow \R \\
		&(\tilde{p}, v) \mapsto \int_{\Omega\setminus A}\underset{i,j}{\sum}   a_{i,j}\partial_i\tilde{p}  \partial_jv + \underset{i}{\sum}d_i( \partial_i \tilde{p} v + \tilde{p} \partial_iv) +  b\tilde{p}v \; dx
		\end{split}
		\end{align}
		being the restriction of bilinear form $a(\cdot,\cdot)$ to $\Omega\setminus A$.
		
		Further, the solutions $p_{c}$ of \cref{adjoint_limitproblem}  converge strongly in $H^1_0(\Omega)$ to the \newline $H^1_0$-representation of $p$.
	\end{theorem}
	
	\begin{proof}
		Let us consider the regularized problem (\ref{SmoothedState}) for $\gamma, c > 0$. Existence and uniqueness of solutions $y_{\gamma, c}, p_{\gamma, c}$ of regularized state and adjoint are guaranteed by application of the Minty-Browder Theorem in analogy to \cref{Main_theorem_1} for $y_{\gamma, c}$ and the Lax-Milgram Theorem for $p_{\gamma, c}$, respectively. 
		
		This proof consists of two main parts:
		\begin{itemize}
			\item[1.] Showing the $H^1$-convergence of the smoothed to the non-smoothed regularized adjoint $p_{\gamma, c} \rightarrow p_c$ for $\gamma \rightarrow \infty$. \item[2.] Analyzing the limit PDE (\ref{adjoint_limitproblem}) for $c\rightarrow \infty$ and proving that $p_c \rightarrow p$ in $H^1(\Omega)$ for $c \rightarrow \infty$, where $p$ is defined as in (\ref{adjoint_unregu}).
		\end{itemize}
		
		\textbf{To 1.} We start to show the $H^1$-convergence of the smoothed to the non-smoothed regularized adjoint $p_{\gamma, c} \rightarrow p_c$ for $\gamma \rightarrow \infty$.
		
		The assumption (\ref{AssumptionOnSign}) of $L^1$-convergence of $\textup{sign}_\gamma(\bar{\lambda} + c\cdot(y_{\gamma, c} - \varphi))$ is equivalent to $L^p$-convergence for all $p\in [1,\infty)$ in our setting, since
		\begin{align*}
		\hspace{-.4cm}\;&\Vert\textup{sign}_\gamma(\bar{\lambda} + c\cdot(y_{\gamma, c} - \varphi)) - \textup{sign}(\bar{\lambda} + c\cdot(y_{c} - \varphi))\Vert_{L^p(\Omega)} \\ 
		\leq &\Vert\textup{sign}_\gamma(\bar{\lambda} + c\cdot(y_{\gamma, c} - \varphi)) - \textup{sign}(\bar{\lambda} + c\cdot(y_{c} - \varphi))\Vert_{L^1(\Omega)}^{1/p} \rightarrow 0 \quad \text{for } \gamma \rightarrow \infty 
		\end{align*}
		by monotony of the integral and \cref{AssumptionsOnMax} (ii)-(iv).
		Denote by $S_{\gamma,c}\colon H_0^1(\Omega) \rightarrow H^{-1}(\Omega)$ the linear operator corresponding to the left-hand side of the smoothed adjoint equation (\ref{SmoothedAdjoint}) and $S_{c}\colon H^1(\Omega) \rightarrow H^{-1}(\Omega)$ the one to (\ref{adjoint_limitproblem}). 
		We establish convergence of $S_{\gamma,c}$ to $S_c$ in the operator norm. 
		In the following, we apply H\"older's inequality and moreover, we use $L^p$-convergence of $\textup{sign}_\gamma(\bar{\lambda} + c\cdot(y_{\gamma, c} - \varphi))$ for all $p\in [1,\infty)$ as well as boundedness of $\textup{sign}_\gamma$ and $\textup{sign}$. 
		
		Further, since we are in the situation $\Omega\subset \R^n$ for $n\leq 4$, we have the following embedding with embedding constant $C>0$ (cf. \cite[Thm. 4.12 Part I, Case C]{adams2003sobolev})
		\begin{align}\label{EmbeddingH1inL4}
		H^1_0(\Omega) &\hookrightarrow L^4(\Omega) 	\quad \text{for } n \leq 4 
		\end{align}
		Combining all this yields
		\begin{align*}
		&\hspace{-.4cm}\Vert S_{\gamma,c} - S_{c} \Vert_{\text{op}} \\
		=
		&\mathop{\sup_{g\in H_0^1(\Omega)}}_{\Vert g \Vert = 1} \mathop{\sup_{h\in H_0^1(\Omega)}}_{\Vert h \Vert = 1} c\cdot\vert ((\textup{sign}_\gamma(\bar{\lambda} + c\cdot(y_{\gamma, c} - \varphi)) - \textup{sign}(\bar{\lambda} + c\cdot(y_{c} - \varphi)))\cdot g,h)_{L^2(\Omega)} \vert \\
		\leq 
		&\mathop{\sup_{g\in H_0^1(\Omega)}}_{\Vert g \Vert = 1} \mathop{\sup_{h\in H_0^1(\Omega)}}_{\Vert h \Vert = 1} c\cdot\Vert \big(\textup{sign}_\gamma(\bar{\lambda} + c\cdot(y_{\gamma, c} - \varphi))\\[-.6cm]
		&\hspace{3.8cm} - \textup{sign}(\bar{\lambda} + c\cdot(y_{c} - \varphi))\big)\Vert_{L^2(\Omega)}\cdot\Vert g \Vert_{L^4(\Omega)}\cdot \Vert h\Vert_{L^4(\Omega)} \\
		\leq
		&\; C^2\cdot c\cdot\Vert \textup{sign}_\gamma(\bar{\lambda} + c\cdot(y_{\gamma, c} - \varphi)) - \textup{sign}(\bar{\lambda} + c\cdot(y_{c} - \varphi))\Vert_{L^2(\Omega)} \rightarrow 0 \quad \text{for } \gamma \rightarrow \infty,
		\end{align*}
		which gives the desired convergence in the operator norm. Using analyticity of the inversion $\mathcal{I}\colon S \mapsto S^{-1}$ in the domain of invertible, bounded, linear operators given in our setting, convergence of the solution operators $S^{-1}_{\gamma,c} \rightarrow S^{-1}_c$ in operator norm is implied immediately, see e.g., \cite[page 237]{zeidler2012applied}.
		Combining this with the convergence of $y_{\gamma, c} \rightarrow y_c$ in $H^1_0(\Omega)$ established by \cref{Main_theorem_1} yields
		\begin{align*}
		&\hspace{-.3cm}\Vert p_{\gamma,c} - p_c \Vert_{H^1_0(\Omega)} \\
		=
		&\Vert -S^{-1}_{\gamma,c}(y_{\gamma,c} - \bar{y}) + S^{-1}_{c}(y_{c} - \bar{y}) \Vert_{H^1_0(\Omega)}  \\
		\leq 
		&\Vert S^{-1}_{\gamma,c}(y_{\gamma,c} - \bar{y}) - S^{-1}_{\gamma,c}(y_c - \bar{y}) \Vert_{H^1_0(\Omega)} + \Vert S^{-1}_{\gamma,c}(y_c - \bar{y})- S^{-1}_{c}(y_{c} - \bar{y}) \Vert_{H^1_0(\Omega)} \\
		\leq
		& \Vert S^{-1}_{\gamma,c} \Vert_{\text{op}} \Vert y_{\gamma,c} - y_c \Vert_{H^1_0(\Omega)} +
		\Vert S^{-1}_{\gamma,c} - S^{-1}_{c} \Vert_{\text{op}} \Vert y_c - \bar{y}\Vert_{H^1_0(\Omega)} \rightarrow 0 \quad \text{for } \gamma \rightarrow \infty,
		\end{align*}
		since $\Vert S^{-1}_{\gamma,c} \Vert_{\text{op}}$ can be bounded due to convergence.
		
		\textbf{To 2.} Next, we analyze the limit PDE (\ref{adjoint_limitproblem}) for $c\rightarrow \infty$. We show that $p_c \rightarrow p$ in $H^1(\Omega)$ for $c \rightarrow \infty$, where $p$ is defined as in (\ref{adjoint_unregu}).
		For this, we first notice that our assumption concerning regular decomposability of $A = \{ x\in \Omega \;\vert\; y -\varphi \geq 0 \}$ ensures that $\partial A =\{ x\in \Omega \;\vert\; y - \varphi = 0 \}$ forms a $C^{0,1}$-manifold embedded in  $\Omega$. This in turn leads to well definedness of the restricted bilinear form $a_{\Omega\setminus A}(.,.)$ and the well-posedness of the variational problem (\ref{adjoint_subproblem}) and, thus, of $p\in H^1_0(\Omega)$.
		Our next step is to show
		\begin{align}\label{WeakConvActiveSet}
		p_c \rightarrow p \quad \text{in } H^1(\Omega) \text{ for } c \rightarrow \infty. 
		\end{align}
		To show this, we artificially constrain problem (\ref{adjoint_limitproblem}) to $A\subseteq \Omega$. So denote by $a_A(\cdot,\cdot)$ the restriction of the bilinear form $a(\cdot,\cdot)$ to $A\subseteq \Omega$, defined in analogy to \cref{BilinRestricted}. The corresponding restricted problem becomes
		\begin{equation}\label{adjoint_subproblem_A}
		a_A(p_{c_{\vert A}}, v) + c\cdot \int_{A} \mathbbm{1}_{A_c} \cdot p_{c_{\vert A}} \cdot v \;dx = - \int_{A} (y_c - \bar{y})\cdot v \;dx \quad \forall v \in H^1_0(A),
		\end{equation}
		where the Dirichlet condition $p_{c_{\vert A}} = p_c$ on $\partial A$ is incorporated in the usual way. 
		Dividing by $c > 0$ gives an equivalent equation in the sense that a solution $p_{c_{\vert A}} \in H^1(A)$ to \cref{adjoint_subproblem_A} also solves the equivalent equation
		\begin{equation}\label{adjoint_subproblem_A_equiv}
		\frac{1}{c}\cdot a_A(p_{c_{\vert A}}, v) + \int_{A} \mathbbm{1}_{A_c} \cdot p_{c_{\vert A}} \cdot v \;dx = - \frac{1}{c} \cdot \int_{A} (y_c - \bar{y})\cdot v \;dx \quad \forall v \in H^1_0(A).
		\end{equation}
		The differential operator corresponding to the left-hand side of the equivalent equation (\ref{adjoint_subproblem_A_equiv}) is given by :
		\begin{equation}
		S_{A,c}\colon H^1(A) \rightarrow H^{-1}(A),\; p \mapsto \frac{1}{c}\cdot a_A(p, \cdot) + ( \mathbbm{1}_{A_c} \cdot p, \cdot )_{L^2(A)}.
		\end{equation}
		Next, we show that the differential operators $S_{A,c}$ converge in the linear operator norm $\Vert \cdot \Vert_{\text{op}}$ with the limit operator
		\begin{equation}
		S_A\colon H^1(A) \rightarrow H^{-1}(A),\; p \mapsto (p, \cdot )_{L^2(A)}.
		\end{equation}
		
		Whence 
		\begin{align*}
		&	\hspace*{-.3cm}\Vert S_{A,c} - S_A \Vert_{\text{op}}\\
		=& \mathop{\sup_{g\in H_0^1(A)}}_{\Vert g \Vert = 1} \mathop{\sup_{h\in H_0^1(A)}}_{\Vert h \Vert = 1} \left\vert \frac{1}{c}\cdot a_A(g, h) - \int_{A \setminus A_c} g\cdot h \;dx \right\vert \\
		\leq & \mathop{\sup_{g\in H_0^1(A)}}_{\Vert g \Vert = 1} \mathop{\sup_{h\in H_0^1(A)}}_{\Vert h \Vert = 1}
		\Bigg(\frac{1}{c} \Big(\underset{i,j}{\sum}\Vert a_{i,j}\Vert_{L^\infty(\Omega)} + \underset{j}{\sum} \Vert d_j \Vert_{L^\infty(\Omega)} + \Vert b \Vert_{L^\infty(\Omega)}\Big)\\[-.2cm]
		&\hspace{3.3cm} \cdot \Vert g \Vert_{H^1_0(A)}\cdot \Vert h \Vert_{H^1_0(A)} \\
		&\hspace{3.3cm}+ \text{vol}(A\setminus A_c)^{\frac{1}{2}} \cdot \Vert g  \Vert_{L^4(A)}\cdot \Vert h \Vert_{L^4(A)} \Bigg) \\
		=&\, \frac{1}{c} \Big(\underset{i,j}{\sum}\Vert a_{i,j}\Vert_{L^\infty(\Omega)} + \underset{j}{\sum} \Vert d_j \Vert_{L^\infty(\Omega)} + \Vert b \Vert_{L^\infty(\Omega)}\Big) + C^2\cdot \text{vol}(A\setminus A_c)^{\frac{1}{2}} \\
		&\hspace{-.3cm} \rightarrow 0 \quad \text{ for } c \rightarrow \infty,
		\end{align*}
		due to \cref{EmbeddingH1inL4} and since $\text{vol}(A\setminus A_c) \rightarrow 0 \text{ for } c \rightarrow \infty$, which would otherwise contradict $y_c \rightarrow y$ in $H^1_0(\Omega)$.
		We can now apply a similar argument as in Step 1, namely the analyticity of the inversion operator $\mathcal{I}\colon S \mapsto S^{-1}$, giving us convergence of the solution operators $S_{A,c}^{-1} \rightarrow S_A^{-1}$ in $\Vert \cdot \Vert_{\text{op}}$. Also notice that we can obtain the sequence of solutions $p_{c_{\vert A}}$ by solving \cref{adjoint_subproblem_A_equiv} with the corresponding right hand sides $-\frac{1}{c}(y_c - \bar{y})$ instead of the original equation (\ref{adjoint_subproblem_A}) and that the right hand sides converge to $0$ in $H^1(\Omega)$ as $c \rightarrow \infty$, as $y_c$ is convergent by \cref{Theorem0}. We conclude
		\begin{align*}
		0 &\leq \Vert p_c \Vert_{H^1(A)} = \Vert S_{A,c}^{-1}\big(-\frac{1}{c}(y_c - \bar{y})\big) \Vert_{H^1(A)} \\
		&\leq \frac{1}{c}\Vert (S_{A,c}^{-1} - S_A^{-1})(y_c - \bar{y})\Vert_{H^1(A)} + \frac{1}{c}\Vert S_A^{-1}(y_c - \bar{y}) \Vert_{H^1(A)} \\
		&\leq \frac{1}{c} \big(\Vert S_{A,c}^{-1} - S_A^{-1} \Vert_{\text{op}} + \Vert S_A^{-1} \Vert_{\text{op}}\big)  \cdot \big(\Vert y_c - y \Vert_{H^1(A)} + \Vert \bar{y} \Vert_{H^1(A)} \big)\\
		&\rightarrow 0 \quad \text{for } c \rightarrow \infty.
		\end{align*}

		For the proof of convergence it remains to address the convergence of $p_c$ on $\Omega\setminus A$.
		We can artificially restrict \cref{adjoint_limitproblem} to
		$\Omega\setminus A$ by imposing the Dirichlet boundary $p_{c\vert A}$ on $\partial A$, since $\partial A$ forms a $C^{0,1}$-submanifold of $\Omega$ as we assumed regular decomposability (cf.~\cref{RegDecompoDefi}) of the active set $A$. To distinguish the corresponding bilinear forms, we denote the restricted bilinear form by $a_{\Omega\setminus A}$.
		Since the unrestricted bilinear form $a(\cdot,\cdot)$ is strongly elliptic, coercivity for some constant $K>0$ also holds for $a_{\Omega\setminus A}$.
		This together with H\"{o}lders inequality, assumption $A_c \subseteq A$ for all $c>0$ and the fact that $p_c - \tilde{p}\in H^1_0(\Omega \setminus A)$ can act as a testfunction gives
		\begin{align*}
		0 &\leq K\Vert p_c - \tilde{p}\Vert^2_{H^1(\Omega\setminus A)} \leq a_{\Omega\setminus A}(p_c - \tilde{p}, p_c - \tilde{p}) = a_{\Omega\setminus A}(p_c, p_c - \tilde{p}) - a_{\Omega\setminus A}(\tilde{p}, p_c  - \tilde{p}) \\
		&= -c\int_{\Omega\setminus A}\mathbbm{1}_{A_c}p_c (p_c - \tilde{p}) \; dx - \int_{\Omega\setminus A}(y_c - \bar{y})(p_c - \tilde{p}) \;dx 
		+ \int_{\Omega\setminus A}(y - \bar{y})(p_c - \tilde{p}) \;dx \\
		&= \int_{\Omega\setminus A}(y - y_c)(p_c - \tilde{p}) \;dx 
		\leq \Vert y_c - y \Vert_{H^1(\Omega)} \Vert p_c - \tilde{p} \Vert_{H^1(\Omega\setminus A)}, 
		\end{align*}
		where $\tilde{p} \in H^1(\Omega\setminus A)$ is defined as in 
		(\ref{adjoint_subproblem}). This results in 
		\begin{align}
		p_c \rightarrow p \quad \text{in } H^1_0(\Omega\setminus A) \quad \text{for } c\rightarrow \infty
		\end{align}
		due to our assumptions and $y_c \rightarrow y$ in $H^1(\Omega)$ as by \cref{Theorem0}. Together with \cref{WeakConvActiveSet} this gives the desired convergence $p_c \rightarrow p$ in $H^1_0(\Omega)$.
	\end{proof}
	
	There are a few non-trivial assumptions in \cref{Main_theorem_1}: assumption (iii) and (iv). In the following, we formulate two remarks in which we address these assumptions (cf.~\cref{FeasabilityKunish} for (iii) and \cref{condition_v} for (iv)).
	
	\begin{remark}\label{FeasabilityKunish}
		It is possible to fulfill assumption (\ref{ActiveSetAssump}) on inclusion of the active sets $A_c\subset A$ by choosing a sufficient $\bar{\lambda}\in L^2(\Omega)$. To be more precisely, if we assume $\varphi\in H^2(\Omega)$, we can choose $\bar{\lambda}:= \max\{0, f - S\varphi\}$ with $S$ being the differential operator corresponding to the elliptic bilinear form $a(\cdot,\cdot)$ in (\ref{PDE}), guaranteeing feasibility $y_{c_1} \leq y_{c_2} \leq y \leq \varphi$ for all $0< c_1 \leq c_2$. For the proof of this, we refer  to \cite[Section 3.2]{ItoKunisch_VI}.
	\end{remark}
	
	\begin{remark}
		\label{condition_v}
		Assumption~\ref{AssumptionOnSign} ensures that convergence of $\textup{sign}_\gamma$ is compatible with convergence of $y_{\gamma,c}$ for $\gamma \rightarrow \infty$. For giving a working example, we verify this assumption in the numerical section under \cref{NumConvSign} for several demonstrative cases. 
	\end{remark}
	
	\begin{remark}
		The limit object $p\in H^1_0(\Omega)$ of the adjoints $p_{\gamma,c}$ as defined in (\ref{adjoint_unregu}) is  the solution of an elliptic problem (\ref{adjoint_subproblem}) on a domain $\Omega \setminus A$ with topological dimension greater than $0$. This can be exploited in numerical computations, for instance by a fat boundary method for finite elements on domains with holes as proposed by the authors of  \cite{maury2001fat}. 
	\end{remark}
	
	\begin{remark}
		We remind the reader, that $p$ is not necessarily an adjoint to the original problem \cref{objective} constrained by \cref{PDE}, \cref{VI_conditions}, but merely solution of a part of the limit of the optimality conditions for the regularized problem. For a discussion of a similar phenomenon in context of optimal control, we refer the interested reader to \cite[Section 4.2]{christof2017optimal}.
	\end{remark}
	
	
	\subsection{Shape derivatives}
	In this section, we apply our convergence results for the regularized state and adjoint equations to derive similar convergence results for the shape derivatives of the shape optimization problem constrained by the fully regularized state equation (\ref{SmoothedState}). 
	In general, shape derivatives of the unregularized VI constrained shape optimization problems do not exist (cf., e.g., \cite[Chapter 1.1]{SokoZol}). Nevertheless, we show existence of  shape derivatives for the shape optimization problem constrained by the fully regularized VI \cref{SmoothedState}. Then, a limiting object corresponding to the unregularized equation \cref{PDE}-\cref{VI_conditions} is derived.
	
	In the following, we split the main results into two theorems, the first one being the shape derivative for the fully regularized equation, the second one being convergence of the former for $\gamma, c \rightarrow \infty$.
	
	The shape derivative of a general shape functional $H$ at $\Omega$ in direction of a sufficiently smooth vector field $V$ is denoted by $DH(\Omega)[V]$. For the definition of shape derivatives or a detailed introduction into shape calculus, we refer to the monographs \cite{Delfour-Zolesio-2001,SokoZol}.
	In general, we have to deal with so-called material and shape derivatives of generic functions $h\colon \Omega\to \R$ in order to derive shape derivatives of objective shape functions. For their definitions and more details we refer to the literature, e.g., \cite{SokoNov}. In the following, we denote the material derivative of $h$ by $\dot{h}$ or $D_m(h)$ and the shape derivative of $h$ in the direction of a vector field $V$ is denoted by $h'$.
	
	\begin{remark}
		\label{Remark_shapederiv}
		In this section, we only consider the shape functional $J$ defined in (\ref{objective}) without regularization term $\mathcal{J}_{\text{\emph{reg}}}$, i.e., we focus only on $\mathcal{J}$. The shape derivative of $J$ is given by the sum of the shape derivative of $\mathcal{J}$ and $\mathcal{J}_{\text{\emph{reg}}}$, where $D\mathcal{J}_{\text{\emph{reg}}}(\Omega)[V]= \nu\int_{\Gamma_{\text{\emph{int}}}}\kappa\left<V,n\right> ds$ with $\kappa:=\text{\emph{div}}_{\Gamma_{\text{\emph{int}}}}(n)$ denoting the mean curvature of $\Gamma_{\text{\emph{int}}}$. Please note that the objective functional and the shape derivative in correlation with the regularized VI (\ref{SmoothedState}) depends on the parameters $\gamma$ and $c$. In order to denote this dependency, we use the notation $\mathcal{J}_{\gamma, c}$ and $D\mathcal{J}_{\gamma, c}(\Omega)[V]$ for the objective functional and its shape derivative, respectively.
	\end{remark}
	
	
	We state the first theorem, which presents the shape derivative of the objective functional $\mathcal{J}$ defined in (\ref{objective}) constrained by the fully regularized VI (\ref{SmoothedState}).
	
	\begin{theorem}
		\label{MainTheorem2}
		Assume the setting of the shape optimization problem formulated in \cref{section_ModelProblem}. Let the assumptions of \cref{Main_theorem_1} hold. Moreover, let $M := (a_{i,j})_{i,j = 1,2,\dots,n}$ be the matrix of coefficient functions to the leading order terms in (\ref{bilinearform}). Assume $y_{\gamma,c},p_{\gamma,c} \in W^{1,4}(\Omega)$, $a_{ij}, b_i, d \in L^\infty(\Omega)\cap W^{1,4}(\Omega)$ and $f\in H^1(\Omega)$. Furthermore, let $D_m(y_{\gamma,c}), D_m(p_{\gamma,c}) \in H^1_0(\Omega)$ for all $\gamma,c >0$.
		Then the shape derivatives of $\mathcal{J}$ defined in (\ref{objective}) constrained by a fully regularized VI (\ref{SmoothedState}) in direction of a vector field $V\in H^1_0(\Omega, \R^n)$ exist and are given by
		\begin{align}\label{shape_deriv}
		\begin{split}
		& D\mathcal{J}_{\gamma, c}(\Omega)[V] \\ &= \int_{\Omega}-(y_{\gamma, c}-\bar{y})\nabla\bar{y}^T V - \nabla y_{\gamma,c}^T (\nabla V^T M - \nabla M\cdot V + M^T\nabla V   ) \nabla p_{\gamma,c} \\
		&\hspace{.9cm}+ (\nabla b^TV)y_{\gamma,c}p_{\gamma,c}+ y_{\gamma,c}\cdot((\nabla d^T V)^T\nabla p_{\gamma,c} - d^T(\nabla V\nabla p_{\gamma,c}))\\
		&\hspace{.9cm}	+ p_{\gamma,c}\cdot((\nabla d^T V)^T\nabla y_{\gamma,c} - d^T(\nabla V\nabla y_{\gamma,c})) \\
		& \hspace{.9cm}- c\cdot \textup{sign}_{\gamma}(\bar{\lambda} + c\cdot(y_{\gamma, c} - \varphi))\cdot\nabla \varphi^T V\cdot p_{\gamma,c} - \nabla f^TVp_{\gamma, c} \\
		&\hspace{.9cm}
		+ \textup{div} (V)\Big( \frac{1}{2}(y_{\gamma, c} -\bar{y})^2 + by_{\gamma,c} p_{\gamma,c} + \underset{i,j}{\sum}   a_{i,j}\partial_i y_{\gamma,c} \partial_j p_{\gamma,c}  \\
		& \hspace{2.6cm} + \sum_id_i (\partial_iy_{\gamma,c}p_{\gamma,c}+ y_{\gamma,c}\partial_i p_{\gamma,c}) \\
		& \hspace{2.6cm}+ \textup{max}_{\gamma}\big(\bar{\lambda} + c\cdot(y_{\gamma, c} - \varphi)\big)p_{\gamma, c} - fp_{\gamma, c} \Big) \; dx \,.
		\end{split}
		\end{align}
	\end{theorem}
	
	\begin{proof}
		Let us consider the shape optimization problem with fully regularized state equations with parameters $\gamma, c >0 $ as in (\ref{SmoothedState}) and fixed shape  $\Gi$ to derive corresponding shape derivative. 
		The first part of the proof, consisting of the existence of shape derivatives $D\mathcal{J}_{\gamma,c}$ for all $\gamma,c>0$, is found in \cref{Proof_ExistenceShapeDeriv}.	
		
		As the second part of the proof, we derive the shape derivative expression. Note that \cref{shape_deriv} and following integrals are well defined due to our assumptions on integrability combined with  \cref{EmbeddingH1inL4} and dimension $n\leq4$. By applying standard shape calculus techniques (cf.~\cite{berggren2010unified,welker2016efficient}) to the target functional part of the Lagrangian we get
		\begin{align}\label{ShapeDirBilin1}
		\begin{split}
		&D\Big(\frac{1}{2}\int_{\Omega} (y_{\gamma, c} -\bar{y})^2dx\Big)[V]\\&= \int_{\Omega} (y_{\gamma, c} -\bar{y})(D_m(y_{\gamma,c}) - D_m(\bar{y})) +  \frac{1}{2}\textup{div}(V)(y_{\gamma, c} -\bar{y})^2dx \\
		&= \int_{\Omega} (y_{\gamma, c} -\bar{y})D_m(y_{\gamma,c})dx + \int_{\Omega} -(y_{\gamma, c} -\bar{y})\nabla \bar{y}^TV + 
		\frac{1}{2}\textup{div}(V)(y_{\gamma, c} -\bar{y})^2 dx,
		\end{split} 
		\end{align}
		since the target $\bar{y}\in L^2(\Omega)$ does not depend on the shape.
		Next, as similarly found in, e.g., \cite{welker2016efficient}, we calculate the shape derivative of the bilinear form $a(\cdot,\cdot)$. For avoiding confusion with the active sets $A$ and $A_c$, we  call the coefficient matrix $(a_{i,j})_{i,j}$ of the leading order parts of the bilinear form $M$. As before we have
		\begin{align}
		\begin{split}
		D(a(y_{\gamma,c}, p_{\gamma,c}))[V] =& \int_{\Omega}D_m(a(y_{\gamma,c}, p_{\gamma,c})) + \textup{div}(V)\Big(\underset{i,j}{\sum}   a_{i,j}\partial_iy_{\gamma,c}\partial_j p_{\gamma,c} \\
		&+ \sum_i d_i (\partial_i y_{\gamma,c} p_{\gamma,c} + y_{\gamma,c}\partial_i p_{\gamma,c}) + by_{\gamma,c}p_{\gamma,c}\Big)dx.
		\end{split}
		\end{align}
		We use linearity, chain rules, product rules and gradient identities for the material derivative $D_m(\cdot)$, as found in \cite{berggren2010unified}, to reformulate $D_m\big(a(y_{\gamma,c}, p_{\gamma,c})\big)$. For readability, we analyze each term individually. We start with the leading order terms:
		\begin{align*}
		& \hspace{-.3cm}D_m\Big(\underset{i,j}{\sum}  a_{i,j}\partial_iy_{\gamma,c}\partial_jp_{\gamma,c}\Big) \\ =& \underset{i,j}{\sum} D_m(a_{i,j})\partial_iy_{\gamma,c}\partial_jp_{\gamma,c} + a_{i,j}D_m(\partial_iy_{\gamma,c})\partial_jp_{\gamma,c} + a_{i,j}\partial_iy_{\gamma,c}D_m(\partial_jp_{\gamma,c}) \\
		=&
		\underset{i,j}{\sum} D_m(a_{i,j})\partial_iy_{\gamma,c}\partial_jp_{\gamma,c} + a_{i,j}\Big((\partial_i D_m(y_{\gamma,c}) - \underset{k}{\sum} \partial_k y_{\gamma,c}\partial_i V_k )\partial_jp_{\gamma,c} \\
		&\quad+ \partial_iy_{\gamma,c}(\partial_j D_m(p_{\gamma,c}) - \underset{k}{\sum} \partial_k p_{\gamma,c}\partial_j V_k )\Big) \\
		=&
		\underset{i,j}{\sum}\Big( \nabla a_{i,j}^TV\partial_iy_{\gamma,c}\partial_jp_{\gamma,c} + a_{i,j}\partial_i D_m(y_{\gamma,c})\partial_j p_{\gamma,c} + a_{i,j}\partial_i y_{\gamma,c}\partial_j D_m(p_{\gamma,c}) \\
		& \quad-
		a_{i,j} (\partial_iV^T\nabla y_{\gamma,c}) \partial_j p_{\gamma,c} - a_{i,j} \partial_i y_{\gamma,c}(\partial_jV^T\nabla p_{\gamma,c})\Big)	\\
		=&
		\nabla y_{\gamma,c}^T(\nabla M^T V)\nabla p_{\gamma,c} + \underset{i,j}{\sum}\Big(a_{i,j}\partial_i D_m(y_{\gamma,c})\partial_j p_{\gamma,c} + a_{i,j}\partial_i y_{\gamma,c}\partial_j D_m(p_{\gamma,c})\Big) \\
		&-\nabla y_{\gamma,c}^T(\nabla V^T M)\nabla p_{\gamma,c} -\nabla y_{\gamma,c}^T(M\nabla V )\nabla p_{\gamma,c}
		\end{align*}
		For the first order terms of $a(\cdot,\cdot)$ we only compute  $y_{\gamma,c} d^T \nabla p_{\gamma,c}$, since calculations are analogous for the second term by switching the roles of $y_{\gamma,c}$ and $p_{\gamma,c}$. We get
		\begin{align*}
		& \hspace{-.3cm}D_m(y_{\gamma,c} d^T \nabla p_{\gamma,c})\\
		=& D_m(y_{\gamma,c}) d^T\nabla p_{\gamma,c} + y_{\gamma,c}\cdot\underset{i}{\sum}(D_m(d_i)\partial_ip_{\gamma,c} + d_i D_m(\partial_i p_{\gamma,c})) \\
		=&
		D_m(y_{\gamma,c}) d^T\nabla p_{\gamma,c} \\&+ \underset{i}{\sum}\Big(y_{\gamma,c}(\nabla d_i^TV)\partial_ip_{\gamma,c} + y_{\gamma,c}d_i\partial_i D_m(p_{\gamma,c}) - \underset{k}{\sum}(y_{\gamma,c}d_i \partial_k p_{\gamma,c}\partial_i V_k)\Big) \\
		=&
		D_m(y_{\gamma,c})d^T\nabla p_{\gamma,c} + y_{\gamma,c}(\nabla d ^T V)^T \nabla p_{\gamma,c} +
		y_{\gamma,c} d^T\nabla D_m(p_{\gamma,c}) - y_{\gamma,c} d^T (\nabla V \nabla p),
		\end{align*}
		where we again use shape independence of the coefficient functions of $a(\cdot,\cdot)$.
		For the term of order zero we apply the product rule for material derivatives and shape independence of coefficient functions:
		\begin{align*}
		D_m(by_{\gamma,c}p_{\gamma,c})= (\nabla b^T V)y_{\gamma,c}p_{\gamma,c} + b D_m(y_{\gamma,c})p_{\gamma,c} + by_{\gamma,c}D_m(p_{\gamma,c}) 
		\end{align*}
		Combining these formulas, plugging them into 
		\cref{ShapeDirBilin1} and collecting all material derivatives of $y_{\gamma,c}$ and $p_{\gamma,c}$ result in the shape derivative of the bilinear form $a(\cdot,\cdot)$:
		\begin{align*}
		\begin{split}
		& \hspace{-.3cm}D\Big(a(y_{\gamma, c}, p_{\gamma,c})\Big)[V] \\=& \quad a(D_m(y_{\gamma,c}), p_{\gamma,c}) + a(y_{\gamma,c}, D_m(p_{\gamma,c}))  \\
		&+ \int_{\Omega} \nabla y_{\gamma,c}^T(\nabla M^T V - \nabla V^T N - M\nabla V)\nabla p_{\gamma,c} \\
		&+ y_{\gamma,c}\cdot((\nabla d^T V)^T\nabla p_{\gamma,c} - d^T(\nabla V\nabla p_{\gamma,c})) \\
		&+ 
		p_{\gamma,c}\cdot((\nabla d^T V)^T\nabla y_{\gamma,c} - d^T(\nabla V\nabla y_{\gamma,c})) 
		+ (\nabla b^TV)y_{\gamma,c}p_{\gamma,c} \\
		&+ \textup{div}(V)\Big(\underset{i,j}{\sum}   a_{i,j}\partial_iy_{\gamma,c}\partial_j p_{\gamma,c} 
		+ \sum_i d_i (\partial_i y_{\gamma,c} p_{\gamma,c} + y_{\gamma,c}\partial_i p_{\gamma,c}) + by_{\gamma,c}p_{\gamma,c}\Big)dx
		\end{split}
		\end{align*}
		The shape derivative of the term including $\max_{\gamma}$ is calculated by chain rule, which is applicable since we assume sufficient smoothness of $\max_{\gamma}$:
		\begin{align*}
		& \hspace{-.3cm}D\Big((\textup{max}_\gamma\big(\bar{\lambda} + c\cdot(y_{\gamma,c} - \varphi)\big),p_{\gamma,c})_{L^2(\Omega)}\Big)[V]\\ =&
		\int_{\Omega}  D_m\Big((\textup{max}_\gamma\big(\bar{\lambda} + c\cdot(y_{\gamma,c} - \varphi)\big),p_{\gamma,c})_{L^2(\Omega)}\Big) \\&\hspace{.5cm}+ \textup{div}(V)\textup{max}_\gamma\big(\bar{\lambda} + c\cdot(y_{\gamma,c} - \varphi)\big)p_{\gamma,c}dx
		\end{align*}
		\begin{align*}
		&\hspace{-.3cm}	D_m\Big((\textup{max}_\gamma\big(\bar{\lambda} + c\cdot(y_{\gamma,c} - \varphi)\big),p_{\gamma,c})_{L^2(\Omega)}\Big)\\ =&
		\big(\textup{sign}_\gamma(\bar{\lambda} + c\cdot(y_{\gamma,c} - \varphi)) D_m(\bar{\lambda} + c\cdot(y_{\gamma,c} - \varphi)), p_{\gamma,c}\big)_{L^2(\Omega)} \\
		&+ \big(\textup{max}_\gamma\big(\bar{\lambda} + c\cdot(y_{\gamma,c} - \varphi)\big),D_m(p_{\gamma,c})\big)_{L^2(\Omega)} \\
		=&
		-c\cdot\big(\textup{sign}_\gamma(\bar{\lambda} + c\cdot(y_{\gamma,c} - \varphi)) \nabla\varphi^TV, p_{\gamma,c}\big)_{L^2(\Omega)} \\
		& + \big(c\cdot\textup{sign}_\gamma(\bar{\lambda} + c\cdot(y_{\gamma,c} - \varphi)) p_{\gamma,c}, D_m(y_{\gamma,c})\big)_{L^2(\Omega)} \\
		&+
		\big(\textup{max}_\gamma\big(\bar{\lambda} + c\cdot(y_{\gamma,c} - \varphi)\big),D_m(p_{\gamma,c})\big)_{L^2(\Omega)},
		\end{align*}
		due to $D_m(\bar{\lambda}) = 0$ and $D_m(\varphi) = \nabla\varphi^TV$, as $\varphi$ is invariant under perturbation of the domain by problem definition.
		The shape derivative of the last term in the Lagrangian (\ref{LagrangianGammaC}) is given by a simple product rule
		\begin{align*}
		D\big( (f, p_{\gamma,c})_{L^2(\Omega)}\big)[V] = (D_m(f), p_{\gamma,c})_{L^2(\Omega)} + \big(f, D_m(p_{\gamma,c})\big)_{L^2(\Omega)}.
		\end{align*}
		We now use the assumptions  $D_m(y_{\gamma,c}), D_m(p_{\gamma,c})\in H^1_0(\Omega)$. 
		If we rearrange the terms with $D_m(y_{\gamma,c})$ and $D_m(p_{\gamma,c})$ acting as test functions and applying the saddle point conditions, which means that the state equation (\ref{SmoothedState}) and adjoint equation (\ref{SmoothedAdjoint}) are fulfilled, the terms consisting $D_m(y_{\gamma,c})$ and $D_m(p_{\gamma,c})$ cancel.  
		By adding all terms of \cref{LagrangianGammaC}, the shape derivative $D\mathcal{J}_{\gamma,c}(\Omega)[V]$ as in (\ref{shape_deriv}) is established. 
	\end{proof}

	\begin{remark}
		One can fulfill the assumptions of the averaged adjoint theorem and, thus, guarantee existence of the shape derivative without a computation of the material derivatives $D_m(y_{\gamma,c}), D_m(p_{\gamma,c})$. The assumption $D_m(y_{\gamma,c}), D_m(p_{\gamma,c}) \in H^1_0(\Omega)$ for all $\gamma,c >0$ in theorem~\ref{MainTheorem2} is only needed in order to calculate the shape derivative expression (\ref{shape_deriv}).
	\end{remark}
	
	\begin{remark}\label{MaterialDerivY}
		The assumption $D_m(y_{\gamma, c})\in H^1_0(\Omega)$ in \cref{MainTheorem2} is needed to apply the saddle point conditions and get the closed form of the shape derivative. For this, it is sufficient that $y_{\gamma,c} \in H^2_0(\Omega)$. For example, this regularity can be ensured by additionally assuming $a_{i,j}\in C^1(\bar{\Omega})$, $d\equiv0$, $b\equiv0$ for the coefficients of the strongly elliptic bilinear form $a(\cdot,\cdot)$ together with $\varphi >0$ and choosing $\bar{\lambda}\in L^\infty(\Omega)$.
		The latter two assumptions imply that the maximal monotone Nemetskii-operator (\ref{NymetskiyGamma}) is equal to $0$ for $y_{\gamma,c} = 0$ and sufficiently large $\gamma, c > 0$. In combination with the former assumptions, \cite[Theorem A.1.]{bonnans1991pontryagin} can be applied to get $y_{\gamma,c}\in H^2_0(\Omega)$ for all sufficiently large $\gamma,c >0$.
	\end{remark}
	
	\begin{remark}\label{MaterialDerivP}
		The assumption $D_m(p_{\gamma, c})\in H^1_0(\Omega)$ in \cref{MainTheorem2} can be fullfilled, e.g., by assuming additional regularity $a_{i,j}\in C^1(\bar{\Omega})$ of the leading coefficients of the bilinear form $a(.,.)$ and $C^2$-regularity of the boundary $\partial \Omega$. This together with the fact that $c\cdot\textup{sign}_\gamma(\bar{\lambda} + c\cdot(y_{\gamma,c} - \varphi)) \in L^\infty(\Omega)$ acts as part of the coefficient function of the zero order terms permits application of a regularity theorem for linear elliptic problems (cf. \cite[p. 317, theorem 4]{Evan}) giving $p_{\gamma,c} \in H^2(\Omega)$. This in turn guarantees $D_m(p_{\gamma, c})\in H^1_0(\Omega)$.
	\end{remark}
	
	Next, we formulate the second main theorem of this section, which states the convergence of the shape derivatives of the fully regularized problem.
	\begin{theorem}\label{MainTheorem3}
		Assume the setting of the shape optimization problem formulated in (\ref{section_ModelProblem}). Let the assumptions of \cref{Main_theorem_1} hold and $\varphi \in H^2(\Omega)$. Moreover, let $M := (a_{i,j})_{i,j = 1,2}$ be the matrix of coefficient functions to the leading order terms in (\ref{bilinearform}).
		Then, for all $V\in H_0^1(\Omega, \R^n)$, the shape derivatives $D\mathcal{J}_{\gamma, c}(\Omega)[V]$ in (\ref{shape_deriv}) converge to $D\mathcal{J}(\Omega)[V]$ for $\gamma, c \rightarrow \infty$, where
		\begin{align}\label{UnreguShapeDeriv}
		\begin{split}
		&\hspace{-1cm} D\mathcal{J}(\Omega)[V] :=\\ \int_{\Omega} & -(y-\bar{y})\nabla\bar{y}^T V - \nabla y^T (\nabla V^T M - \nabla M\cdot V + M^T\nabla V   ) \nabla p \\
		&+ y\cdot((\nabla d^T V)^T\nabla p - d^T(\nabla V\nabla p))
		+ p\cdot((\nabla d^T V)^T\nabla y - d^T(\nabla V\nabla y)) \\
		& + (\nabla b^TV)yp - \nabla f^TVp 
		\\&+ \textup{div} (V)\Big( \frac{1}{2}(y_{\gamma, c} -\bar{y})^2 + \underset{i,j}{\sum}   a_{i,j}\partial_i y \partial_j p \\&+ \sum_id_i (\partial_iyp+ y\partial_i p) 
		+ by p - fp \Big) \; dx 
		+ \int_{A}(\varphi - \bar{y})\nabla \varphi^TV\;dx.
		\end{split}
		\end{align}
	\end{theorem}
	
	\begin{proof}
		We see that (\ref{shape_deriv}) already resembles (\ref{UnreguShapeDeriv}) except for the two terms
		\begin{align}
		\label{1TermThrm2}
		T_1(V)& :=- c\cdot\int_{\Omega}\textup{sign}_{\gamma}(\bar{\lambda} + c\cdot(y_{\gamma, c} - \varphi))\cdot \nabla \varphi^T V\cdot p_{\gamma,c} \;dx\,\\
		\label{2TermThrm2}
		T_2(V) & :=\int_{\Omega}\text{div}(V)\cdot \textup{max}_{\gamma}\big( \bar{\lambda} + c\cdot(y_{\gamma, c} - \varphi)\big)\cdot p_{\gamma, c}\;dx.
		\end{align}	

		We proceed in two steps: First, we show convergence for $T_1$ and $T_2$ as restricted operators on $C^\infty_0(\Omega, \mathbb{R}^n)$. Second, we show that the limiting operators can be continuously extended to $H^1_0(\Omega, \mathbb{R}^n)$.
		
		Let $V\in C^\infty_0(\Omega, \R^n)$. 
		By this, we have $\textup{div}(V)\cdot p_{\gamma,c}, \nabla \varphi ^TV\in H^1_0(\Omega)$ for all $\gamma,c >0$, which enables to use these functions  as test functions for the state and adjoint equations. This leads to
		\begin{align*}
		T_1(V) =&  -c\cdot\int_{\Omega}\textup{sign}_{\gamma}(\bar{\lambda} + c\cdot(y_{\gamma, c} - \varphi))\cdot \nabla \varphi^T V\cdot p_{\gamma,c} \;dx \\
		=& \;
		a(p_{\gamma,c}, \nabla\varphi^TV) + (y_{\gamma,c}-\bar{y}, \nabla\varphi^TV)_{L^2(\Omega)} \\
		\rightarrow& \;
		a(p, \nabla\varphi^TV) + (y-\bar{y}, \nabla\varphi^TV)_{L^2(\Omega)} =: \tilde{T}_1(V) \quad \text{for } \gamma,c\rightarrow \infty
		\end{align*}
		and 
		\begin{align*}
		T_2(V)=&
		\int_{\Omega}\text{div}(V)\cdot \textup{max}_{\gamma}\big( \bar{\lambda} + c\cdot(y_{\gamma, c} - \varphi)\big)\cdot p_{\gamma, c}\;dx \\
		=& \;
		-a(y_{\gamma,c}, p_{\gamma,c}\cdot\textup{div}(V)) + (f, p_{\gamma,c}\cdot\textup{div}(V))_{L^2(\Omega)} \\
		\rightarrow& \;
		-a(y, p\cdot\textup{div}(V)) + (f, p\cdot\textup{div}(V))_{L^2(\Omega)} =: \tilde{T}_2(V) \quad \text{for } \gamma,c\rightarrow \infty
		\end{align*}
		due to \cref{Main_theorem_1} \cref{Theorem0} and our assumption $\varphi \in H^2(\Omega)$.
		
		Next, we lift the convergence from $V\in C^\infty_0(\Omega, \mathbb{R}^n)$ to $H^1_0(\Omega, \mathbb{R}^n)$ by continuous extension. Since $C^\infty_0(\Omega, \mathbb{R}^n)$ is a dense subspace of $H^1_0(\Omega, \mathbb{R}^n)$ and the latter being the completion of the former by the $\Vert \cdot \Vert_{H^1_0(\Omega, \mathbb{R}^n)}$ norm, it is sufficient to show that the limits of $T_1(V_n)$ and $T_2(V_n)$ form a Cauchy sequence for a given  Cauchy sequence $(V_n)_{n\in \mathbb{N}}\subset C^\infty_0(\Omega, \mathbb{R}^n)$ under the $\Vert \cdot \Vert_{H^1_0(\Omega, \mathbb{R}^n)}$-norm.
		So let $(V_n)_{n\in \mathbb{N}}\subset C^\infty_0(\Omega, \mathbb{R}^n)$ with $\Vert V_n - V_m \Vert_{H^1_0(\Omega, \mathbb{R}^n)} \rightarrow 0$ for $m,n \rightarrow \infty$. 
		For the limit of $T_1$ we have
		\begin{align*}
		&\hspace{-.3cm}\vert \tilde{T}_1(V_n) - \tilde{T}_1(V_m) \vert \\
		=&
		\vert a(p, \nabla\varphi^T(V_n - V_m)) + (y-\bar{y}, \nabla \varphi ^T(V_n - V_m))_{L^2(\Omega)} \vert\\
		\leq&
		\Big(\underset{i,j}{\sum}\Vert a_{i,j}\Vert_{L^\infty(\Omega)} + \underset{j}{\sum} \Vert d_j \Vert_{L^\infty(\Omega)} + \Vert b \Vert_{L^\infty(\Omega)}\Big)\\ 
		&\cdot\int_{\Omega}\underset{i,j}{\sum}\partial_i p \partial_j(\nabla \varphi ^T (V_n-V_m)) \\
		&+  \underset{i}{\sum}(\partial_i p \nabla \varphi ^T(V_n - V_m) + p \partial_i \nabla\varphi^T(V_n - V_m)) \\
		&+ p\nabla\varphi^T(V_n - V_m)  dx + \int_{\Omega} (y- \bar{y})\nabla\varphi^T(V_n-V_m)dx\\
		=& 
		\Big(\underset{i,j}{\sum}\Vert a_{i,j}\Vert_{L^\infty(\Omega)} + \underset{j}{\sum} \Vert d_j \Vert_{L^\infty(\Omega)} + \Vert b \Vert_{L^\infty(\Omega)}\Big)\\
		&\cdot\int_{\Omega}\underset{i,j}{\sum}\Big(\partial_i p \big((\partial_j\nabla \varphi)^T (V_n-V_m) + \nabla \varphi^T \partial_j(V_n-V_m)\big)\Big) \\
		&+ p\nabla\varphi^T(V_n - V_m) dx + \int_{\Omega} (y- \bar{y})\nabla\varphi^T(V_n-V_m)dx
		\end{align*}
		\begin{align*}
		\leq&
		\Big(\underset{i,j}{\sum}\Vert a_{i,j}\Vert_{L^\infty(\Omega)} + \underset{j}{\sum} \Vert d_j \Vert_{L^\infty(\Omega)} + \Vert b \Vert_{L^\infty(\Omega)}\Big)\cdot\Vert p\Vert_{H^1_0(\Omega)}\\
		&\cdot C \cdot
		\Big(\underset{i,j}{\sum}\big(
		\Vert(\partial_j\nabla \varphi)^T (V_n-V_m) \Vert_{L^1(\Omega)} + \Vert \nabla \varphi^T \partial_j(V_n-V_m)\Vert_{L^1(\Omega)}\big) \\ &\hspace{.5cm}+ \Vert \nabla\varphi^T(V_n - V_m) \Vert_{L^1(\Omega)}\Big) \\
		&+
		C\cdot\Vert \varphi \Vert_{H^1(\Omega)}\cdot\Vert y-\bar{y}\Vert_{L^2(\Omega)}\cdot\Vert V_n - V_m\Vert_{H^1_0(\Omega, \mathbb{R}^n)}\\
		\leq&
		C\cdot\Big(\underset{i,j}{\sum}\Vert a_{i,j}\Vert_{L^\infty(\Omega)} + \underset{j}{\sum} \Vert d_j \Vert_{L^\infty(\Omega)} + \Vert b \Vert_{L^\infty(\Omega)}\Big)\cdot\Vert p\Vert_{H^1_0(\Omega)} \\
		&\cdot \Vert \varphi\Vert_{H^1(\Omega)}\cdot 9\cdot\Vert V_n - V_m \Vert_{H^1_0(\Omega,\mathbb{R}^n)} \\ 
		&+
		C\cdot\Vert \varphi \Vert_{H^1(\Omega)}\cdot\Vert y-\bar{y}\Vert_{L^2(\Omega)}\cdot\Vert V_n - V_m\Vert_{H^1_0(\Omega, \mathbb{R}^n)} \rightarrow 0 \quad \text{for } m,n \rightarrow \infty.
		\end{align*}
		Here, we use integration by parts, Gauss' Theorem, $p_{\vert \partial \Omega} = 0$, $(V_n - V_m)_{\vert\partial\Omega} = 0$  and $L^2(\Omega) \hookrightarrow L^1(\Omega)$ with constant $C>0$ as in \cref{EmbeddingLp}. Thus,  $(\tilde{T}_1(V_n))_{n\in\mathbb{N}}$ forms a Cauchy sequence and, therefore, gives a value for the continuous extension of $\tilde{T}_1$ for the limit of $V_n$ in $H^1_0(\Omega, \mathbb{R}^n)$. For $T_2$ we use the same techniques, leading to
		\begin{align*}
		&\hspace{-.3cm}\vert \tilde{T}_2(V_n) - \tilde{T}_2(V_m)\vert\\ =&
		\vert -a(y,p\cdot \textup{div}(V_n - V_m)) + (f, p\cdot \textup{div}(V_n - V_m))_{L^2(\Omega)}\vert \\
		\leq & 
		\Big(\underset{i,j}{\sum}\Vert a_{i,j}\Vert_{L^\infty(\Omega)} + \underset{j}{\sum} \Vert d_j \Vert_{L^\infty(\Omega)} + \Vert b \Vert_{L^\infty(\Omega)}\Big) \\
		&\cdot13\cdot C\cdot\Vert y \Vert_{H^1_0(\Omega)} \cdot\Vert p\Vert_{H^1_0(\Omega)} \cdot \Vert V_n - V_m\Vert_{H^1_0(\Omega, \mathbb{R}^n)} \\
		&+ C\cdot \Vert f \Vert_{L^2(\Omega)}\cdot\Vert p\Vert_{H^1_0(\Omega)} \cdot \Vert V_n - V_m\Vert_{H^1_0(\Omega, \mathbb{R}^n)} \rightarrow 0 \quad \text{for } m,n \rightarrow \infty.
		\end{align*}
		With these convergences $T_1, T_2$ converge to the continuously extended limit objects, which we from now on denote by $T_1, T_2$, for all $V\in H^1_0(\Omega, \mathbb{R}^n)$. 
		Next, we simplify the sum of these two limiting objects. Let $V\in C^\infty_0(\Omega, \mathbb{R}^n)$. Then
		\begin{align}\label{FinalIdent}
		\begin{split}
		& \hspace{-.3cm}T_1(V) + T_2(V) \\= & \;a(p, \nabla\varphi^TV) + (y-\bar{y}, \nabla \varphi ^TV)_{L^2(\Omega)} \\ &
		- a(y,p\cdot \textup{div}(V)) + (f, p\cdot \textup{div}(V))_{L^2(\Omega)} \\
		=& a_{\Omega\setminus A}(p, \nabla\varphi^TV) + (y-\bar{y}, \nabla \varphi ^TV)_{L^2(\Omega\setminus A)} 
		\\&+ a_A(p, \nabla\varphi^TV) + (y-\bar{y}, \nabla \varphi ^TV)_{L^2(A)} 
		+ (\lambda, p\cdot \textup{div}(V))_{L^2(\Omega)}	\\
		=&	(\varphi-\bar{y}, \nabla \varphi ^TV)_{L^2(A)},
		\end{split} 
		\end{align}
		where we use the definition of $p$, complementary slackness of $\lambda\in L^2(\Omega)$, test function properties of $\nabla\varphi^TV$ and $p\cdot \textup{div}(V)$, the state and adjoint equations. We apply again a continuity argument to gain this identity for all $V\in H^1_0(\Omega, \mathbb{R}^n)$. We see that the limit object in (\ref{FinalIdent}) is exactly the missing term in the limit of the shape derivatives $D\mathcal{J}_{\gamma,c}(\Omega)[V]$ (cf.~(\ref{UnreguShapeDeriv})).
	\end{proof}
	
	\begin{remark}
		Theorem~\ref{MainTheorem2} and \cref{MainTheorem3} are also valid when $f\in H^1(\Omega)$ or $\varphi\in H^2(\Omega)$ depend explicitly on the shape $\Omega$ with shape derivatives $f', \varphi'\in H^1_0(\Omega)$. Then the shape derivatives need to be modified accordingly by replacing terms including $\nabla f ^T V$ and $\nabla \varphi^TV$ by $\nabla f ^T V + f'$ and $\nabla \varphi^TV + \varphi'$. 
			Further, Theorem~\ref{MainTheorem2} and \cref{MainTheorem3} remain valid for piecewise constant $f\in L^\infty(\Omega)$ depending on the shape $\Omega$ by adjusting the proofs  applying integral splitting techniques as found in \cite[Remark 4.21, Thrm. 4.23]{welker2016efficient}.
	\end{remark}
	
	
	
	\begin{remark}
		It is common knowledge that by pushing the obstacle $\varphi$ to infinity, i.e., $\varphi(x) \uparrow \infty$ for all $x \in \Omega$, the state equation representing the variational inequality (\ref{PDE}) becomes a regular elliptic PDE in weak formulation
		\begin{align*}
		a(y,v) = (f,v)_{L^2(\Omega)} \quad \forall v\in H^1_0(\Omega)
		\end{align*}
		due to (\ref{VI_conditions}). This means that we encounter shape optimization problems with elliptic PDE constraints.
		Formula (\ref{UnreguShapeDeriv}) remains valid by applying $A=\varnothing$, giving a shape derivative for a general elliptic problem.
	\end{remark}
	
	\begin{remark}
		The limiting object \cref{UnreguShapeDeriv} is in general not the shape derivative of the unregularized problem. It can be regarded as part of the limit system arising during convergence of the optimality systems of the fully regularized problem. Finding a framework in shape optimization to describe the type of this limiting object will be part of further research and is beyond the scope of this article. For readers interested in a treatise on limiting systems of optimality systems in non-smooth optimal control we recommend \cite{christof2017optimal}.
	\end{remark}
	
	\begin{remark}\label{RemarkNoShapeDer}
		The limiting objects of the convergence results for adjoint variables (cf.~\cref{Main_theorem_1}) and shape derivatives (cf.~\cref{MainTheorem2}) can be put into relation by conditions resembling C-stationarity, e.g., as found in \cite[Definition~4.1.]{hintermuller2009mathematical}.
		\\
		Using our terminology, it is necessary for C-stationarity conditions to hold that a $\xi \in H^{-1}(\Omega)$ exists such that the adjoint equation can be formulated in the form
		\begin{align}\label{adj_cStatioForm}
		a(p,v) + \langle\xi,v\rangle = -((y - \bar{y}),v)_{L^2(\Omega)}.
		\end{align}
		We can define such a $\xi\in H^{-1}(\Omega)$ by emulating the definition of $p$ in (\ref{adjoint_unregu}), including enforcement of the Dirichlet condition $p=0$ on $\partial A$ with Nitsche's method using boundary terms (cf. \cite{juntunen2009nitsche}).
		The state equation, corresponding complementarity conditions, and the design equation, which in our setting can be viewed as the shape derivative identity (\ref{shape_deriv}), hold in analogy to the cited definition of C-stationarity.
		The remaining conditions
		\begin{align}
		\langle \xi, p\rangle \geq 0 \quad \text{and } \quad p = 0 \;\;a.e. \;\; \text{in } \{\xi > 0\},
		\end{align}
		by the definitions of $\xi$ and $p$, are satisfied as well. It is worth mentioning that---to knowledge of the authors---no type of C-stationarity-like conditions for optimality of VI constrained shape optimization problems have been investigated or defined before. By defining C-stationarity in this context, as outlined above, we can sum up the theorems by stating that the solutions of the regularized equations converge to a C-stationary system.
	\end{remark}

	\section{Algorithmic aspects and numerical investigations}
	\label{section_Algorithmic}
	In this section, we put the theoretical treatise highlighted in the previous section into numerical practice on domains in $\R^2$. We employ a steepest descent algorithm with backtracking linesearch in order to perform the optimization procedures with various regularized as well as unregularized versions of the specialized variational inequality (see (\ref{NumStateEq})). Also, we propose a way to incorporate the unregularized approach in an algorithm and compare it to the different regularizations.
	
	For convenience, we specialize the more general constraint (\ref{PDE}) to a Laplacian version:
	
	\begin{align}\label{NumProblem}
	\begin{split}
	\min_{\Gi}\;\frac{1}{2}\int_{\Omega} \left|y - \bar{y}\right|^2\; dx+\nu \int_{\Gi}& 1\; ds 
	\end{split}	
	\\
	\begin{split}\label{NumStateEq}
	\text{s.t.} \quad \int_{\Omega}\nabla y^T \nabla v \;dx + \langle \lambda, v \rangle &= \int_{\Omega} fv\;dx	\quad \forall v \in H^1_0(\Omega)\\
	\lambda &\geq 0 \quad \text{in } \Omega \\
	y  &\leq \varphi \quad \text{in } \Omega \\
	\lambda(y-\varphi) &= 0 \quad \text{in } \Omega 
	\end{split}
	\end{align}
	We use $\nu = 10^{-5}$ for all computations in this section.
	As the right-hand side of the state equation in \cref{NumProblem} we choose following piecewise constant function $f\in L^\infty(\Omega)$ defined by
	\begin{align}\label{fRightHandSide}
	f(x) = \begin{cases}
	-10 \quad &\text{for } x\in \Omega_{\text{out}} \\
	100 \quad &\text{for } x\in \Omega_{\text{in}}
	\end{cases}.
	\end{align}
	
	For calculations of the smoothed state and adjoint we have to specify $\max_\gamma$ satisfying \cref{AssumptionsOnMax}. For demonstrative purpose, we choose a similar smoothing procedure as in \cite[Section 2]{ItoKunisch_VI}:
	\begin{align}\label{SmoothingMaxExample}
	\begin{split}
	\text{max}_\gamma(x) = \begin{cases}
	\max(0,x) \quad &\text{for } x\in \mathbb{R}\setminus [-\frac{1}{\gamma}, \frac{1}{\gamma}] \\
	\frac{\gamma}{4}x^2 + \frac{1}{2}x + \frac{1}{4\gamma} \quad &\text{else}
	\end{cases}
	\end{split}		
	\end{align}
	A different, more regular smoothing is, e.g., given in \cite[(1.10)]{schiela2013convergence}. Both smoothing techniques mentioned satisfy \cref{AssumptionsOnMax}. 
	For the sake of completeness, we also give the first derivative formula
	\begin{align}
	\textup{sign}_\gamma(x) = \begin{cases}
	0 \quad &\text{for } x\in(-\infty, -\frac{1}{\gamma}) \\
	\frac{\gamma}{2}x + \frac{1}{2} \quad &\text{for } x\in[-\frac{1}{\gamma}, \frac{1}{\gamma}] \\
	1 \quad &\text{for } x\in (\frac{1}{\gamma}, \infty)
	\end{cases}.
	\end{align}
	%

	In this setting, the shape derivative (\ref{UnreguShapeDeriv}) simplifies to 
	\begin{equation}
	\begin{split}
	&\hspace{-.7cm}D\mathcal{J}(\Omega)[V] \\= \int_{\Omega}&-(y-\bar{y})\nabla \bar{y}^TV - \nabla y^T (\nabla V^T + \nabla V)\nabla p  \\ &
	+ \textup{div}(V)\Big(\frac{1}{2}(y-\bar{y})^2 + \nabla y^T \nabla p - fp\Big) \;dx + \int_{A}(\varphi - \bar{y})\nabla \varphi^TV\;dx
	\end{split}
	\end{equation}
	and analogously the shape derivative for the fully regularized equation in \cref{shape_deriv}. Notice that the shape derivative of the perimeter regularization is also included in our computations (cf.~\cref{Remark_shapederiv}).
	
	In the following numerical experiments, we consider two different obstacles:
	\begin{align}
	\varphi_1(x) = 0.5 \quad \text{and}\quad \varphi_2(x) = 5 e^{-x_1 -1}.
	\end{align}
	
	The calculations are performed with Python using the finite element package FEniCS. For detailed informations on FEniCS, we refer to \cite{AlnaesBlechta2015a} and \cite{LoggMardalEtAl2012a}. As initial shape we choose a centered circle with radius $0.15$, illustrated in \cref{ShapeMorphFigure}. The computational grid of the initial shape, which is embedded in the hold-all-domain $(0,1)^2\subset \mathbb{R}^2$, consists of $2\,184$ vertices with $4\,206$ cells, having a maximum cell diameter of 0.0359 and a minimum cell diameter of 0.018. The algorithm employed for the shape optimization is summarized in algorithm~\ref{MainAlgo}. 
	In the following, we describe the algorithm and the chosen parameters in detail.
	
	The target data $\bar{y}\in L^2(\Omega)$ is computed by using the mesh of the target interface to calculate a corresponding state solution of \cref{NumStateEq} by the semi-smooth Newton method proposed in \cite{ItoKunisch_VI}. These are visualized in \cref{TargetStatesFigure} for both obstacles $\varphi_1$ and $\varphi_2$. We apply the same method for calculating state variables $y$ in the unregularized optimization approach.
	
	\begin{figure}
		\vspace{-1cm}
		\hspace{-1cm}
		\includegraphics[scale=.27]{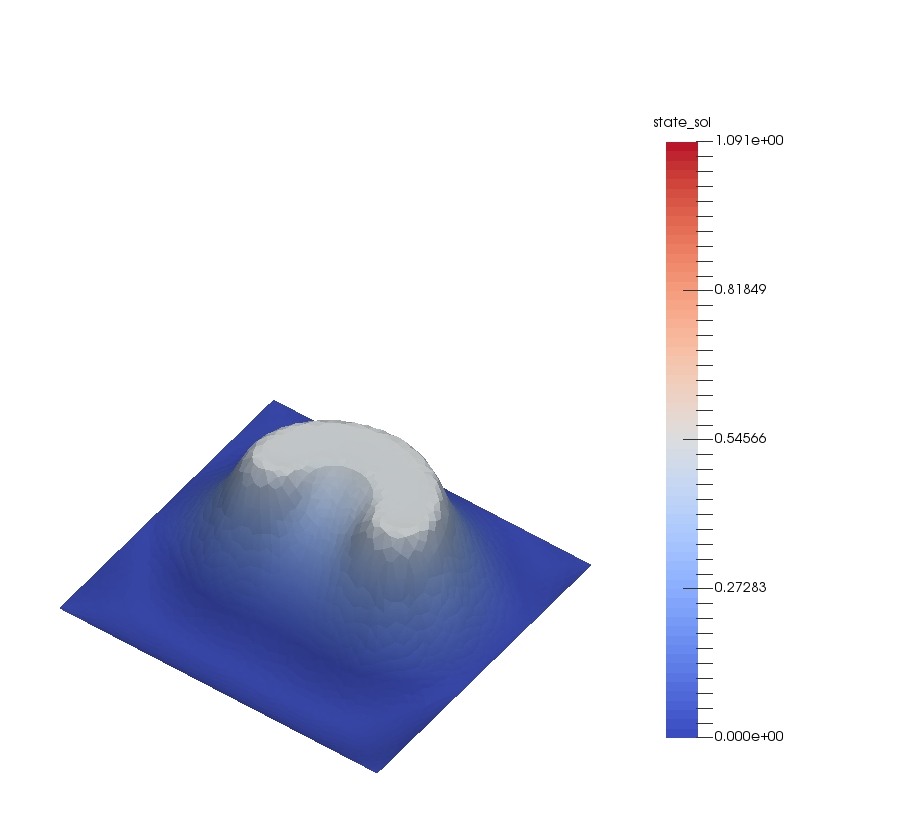}
		\hspace{-1cm}
		\includegraphics[scale=.27]{state_exp_me}
		\vspace{-.2cm}	
		\caption{Solutions $\bar{y}$ to the VI in the target shape. On the left: $\varphi_1 = 0.5$. On the right: $\varphi_2 = 5 e^{-x_1 -1}$.} 
		\label{TargetStatesFigure}
	\end{figure}
	
	For the regularized and smoothed states $y_{\gamma,c}$ and $ y_c$ we use a Newton- and semi-smooth Newton method provided by the FEniCS package in order to solve the linear systems assembled by using first order polynomials on the computational grids. All state calculations in our routines are performed with a stopping criterion of $\varepsilon_{\text{state}} = 3.e-4$ for the error norms.
	In light of \cref{FeasabilityKunish} we choose $\bar{\lambda} = \textup{max}\{0, f + \Delta\varphi\}$, which is possible due to sufficient regularity of $\varphi_1$ and $\varphi_2$. 
	
	To ensure assumptions of \cref{Main_theorem_1}, \cref{MainTheorem2} and \cref{MainTheorem3}, it is necessary to fulfill
	\begin{equation}\label{NumConvSign}
	\Vert\textup{sign}_\gamma(\bar{\lambda} + c\cdot(y_{\gamma, c} - \varphi)) - \textup{sign}(\bar{\lambda} + c\cdot(y_{c} - \varphi))\Vert_{L^1(\Omega)} \rightarrow 0 \quad \text{for } \gamma \rightarrow \infty. 
	\end{equation}
	
	We calculate the corresponding norm using various $c >0$ and both, $\varphi_1$ and $\varphi_2$, on refined meshes having $212\,642$ vertices, $423\,682$ cells and maximum and minimum cell diameter of $0.0038$ and $0.0015$, respectively.
	An example convergence plot can be found in \cref{SignConvFigure}.
	\begin{figure}
		\centering
		\includegraphics[scale=.4]{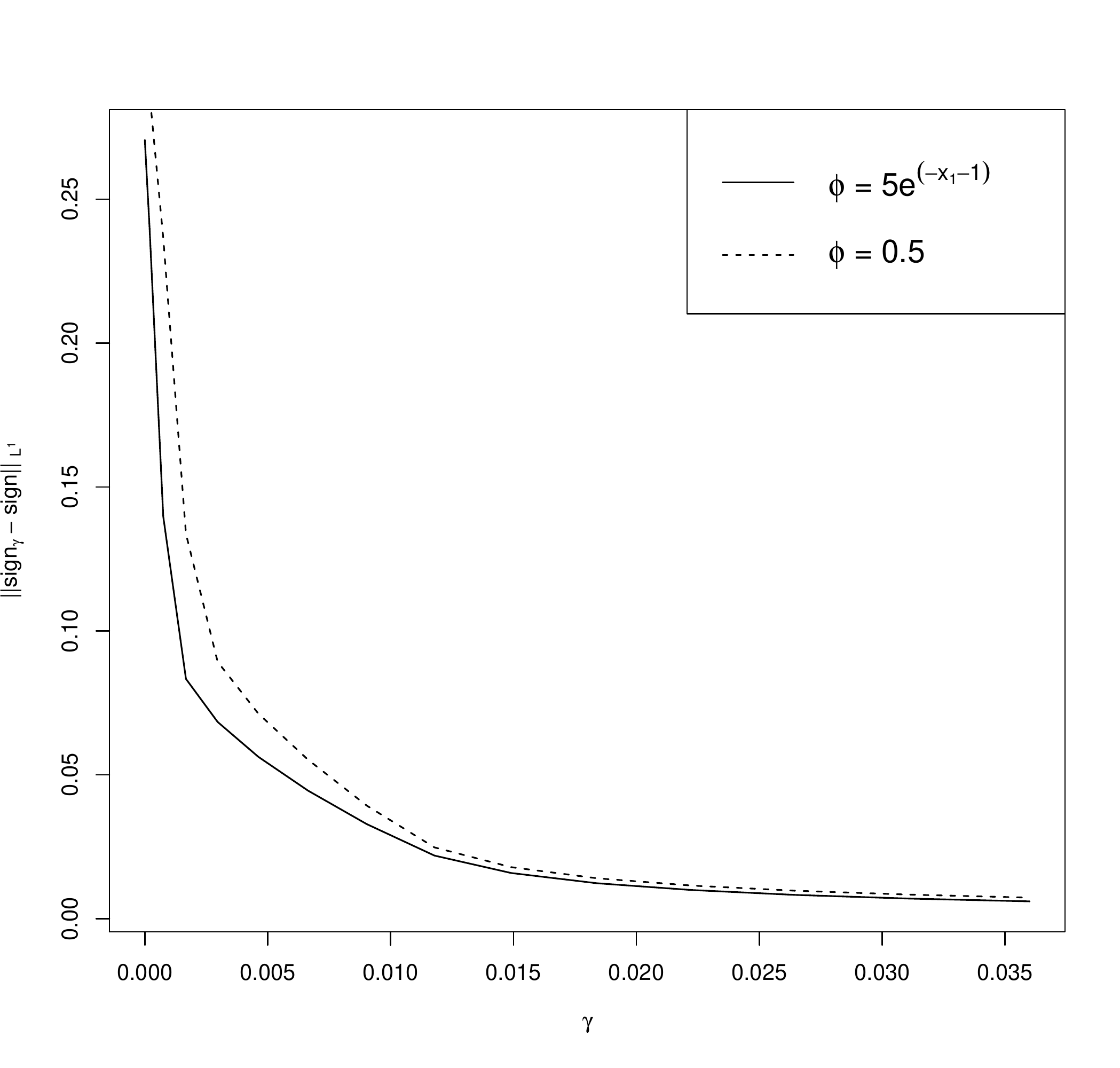}
		\vspace{-.5cm}
		\caption{Convergence plots for $	\Vert\textup{sign}_\gamma(\bar{\lambda} + c\cdot(y_{\gamma, c} - \varphi)) - \textup{sign}(\bar{\lambda} + c\cdot(y_{c} - \varphi))\Vert_{L^1(\Omega)}$ as a function of $\gamma$.}
		\label{SignConvFigure}
	\end{figure}
	\begin{figure}
		\includegraphics[scale=.1282]{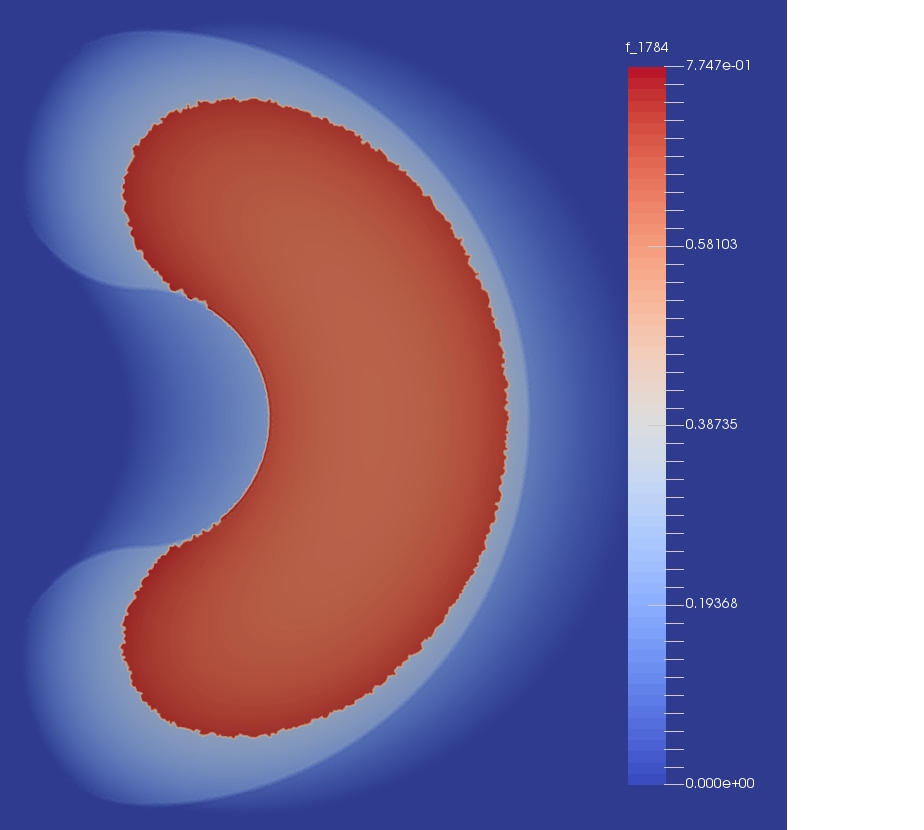}
		\includegraphics[scale=.1282]{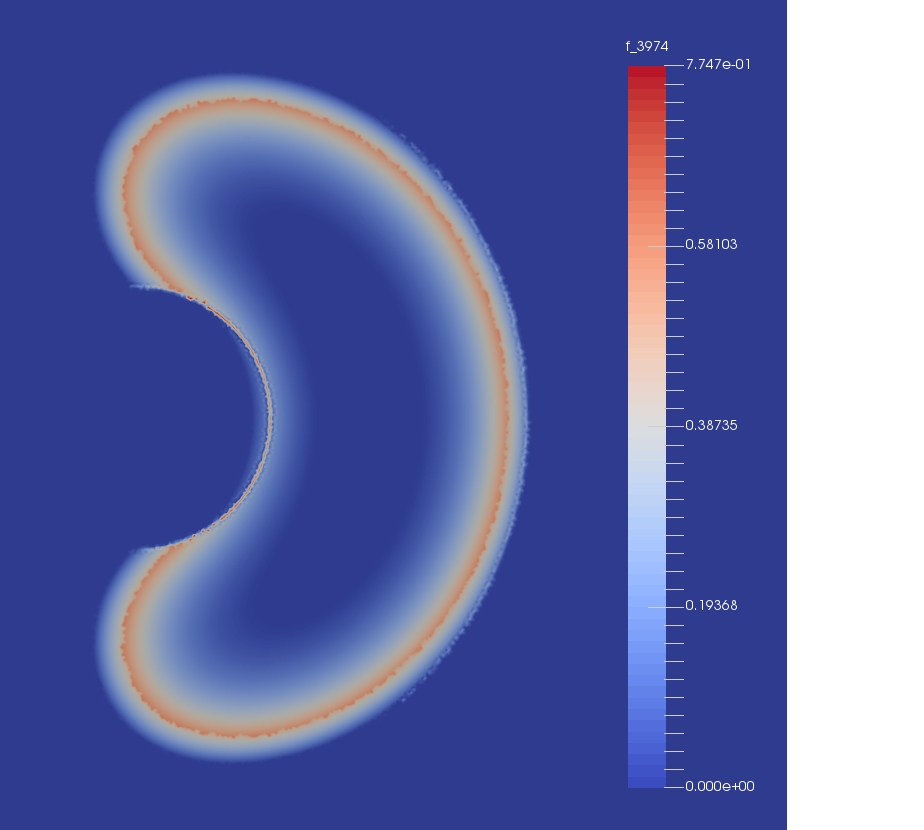} 
		\includegraphics[scale=.1282]{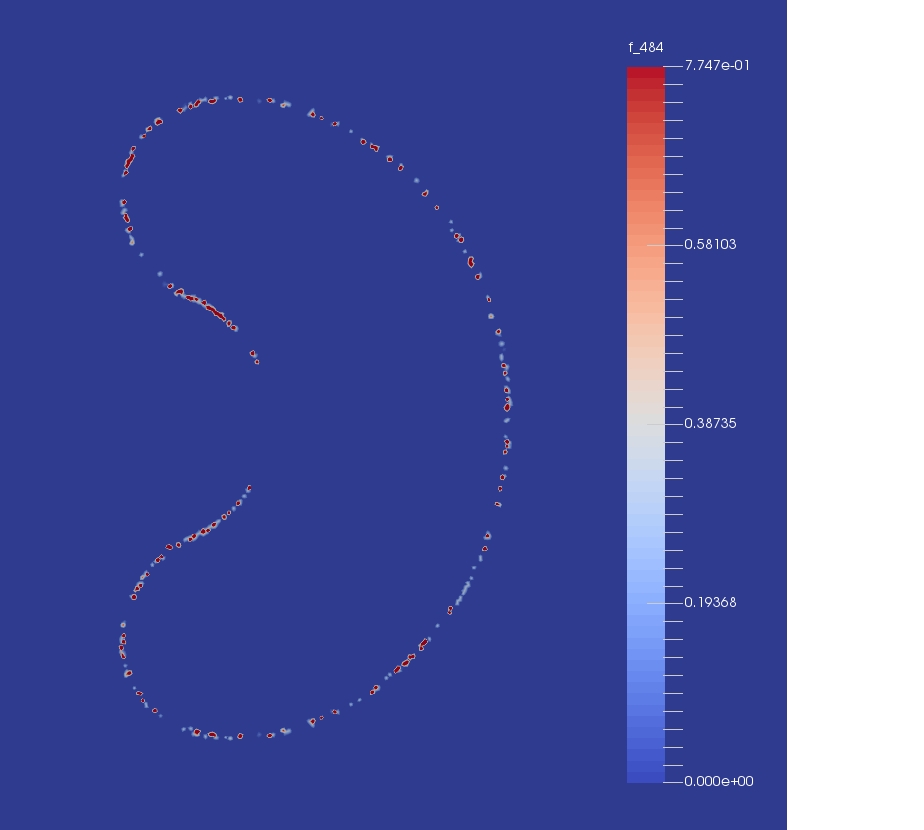}
		\caption{Graphs of $	\textup{sign}_\gamma\big(\bar{\lambda} + c\cdot(y_{\gamma, c} - \varphi)\big) - \textup{sign}\big(\bar{\lambda} + c\cdot(y_{c} - \varphi)\big)$ as functions of $x\in \Omega$ calculated on the refined target mesh with $c = 1\,000$. From left to right: $\gamma = 0.00075$, $\gamma = 0.009$ and $\gamma = 10$.}
		\label{SignDifferenceFigure}
	\end{figure} 
	We want to point out that as $\gamma \rightarrow \infty$, the norm in (\ref{NumConvSign}) converges to an $\varepsilon >0$ which is close to $0$. This is due to numerical errors resulting from the state equation, since their solution determines the active set, which is needed to calculate the values of $\textup{sign}$ and $\textup{sign}_\gamma$. The functions, whose $L^1$-norms are of interest, are illustrated in \cref{SignDifferenceFigure} on a refined mesh. Furthermore, we observe that these functions, and hence the errors go to $0$ for ever finer grid widths and more precisely calculated states $y_{\gamma,c}, y$. This is supported by a study successively evaluating the mentioned $L^1$-norms in the circular start shape, as depicted in \cref{ShapeMorphFigure}, on meshes generated by adaptive refinement at the boundary of the active set for large $\gamma,c >0$. The results are seen in \cref{SignDifferenceStudyPlot}.
	
	\begin{figure}
		\centering
		\includegraphics[scale=.4]{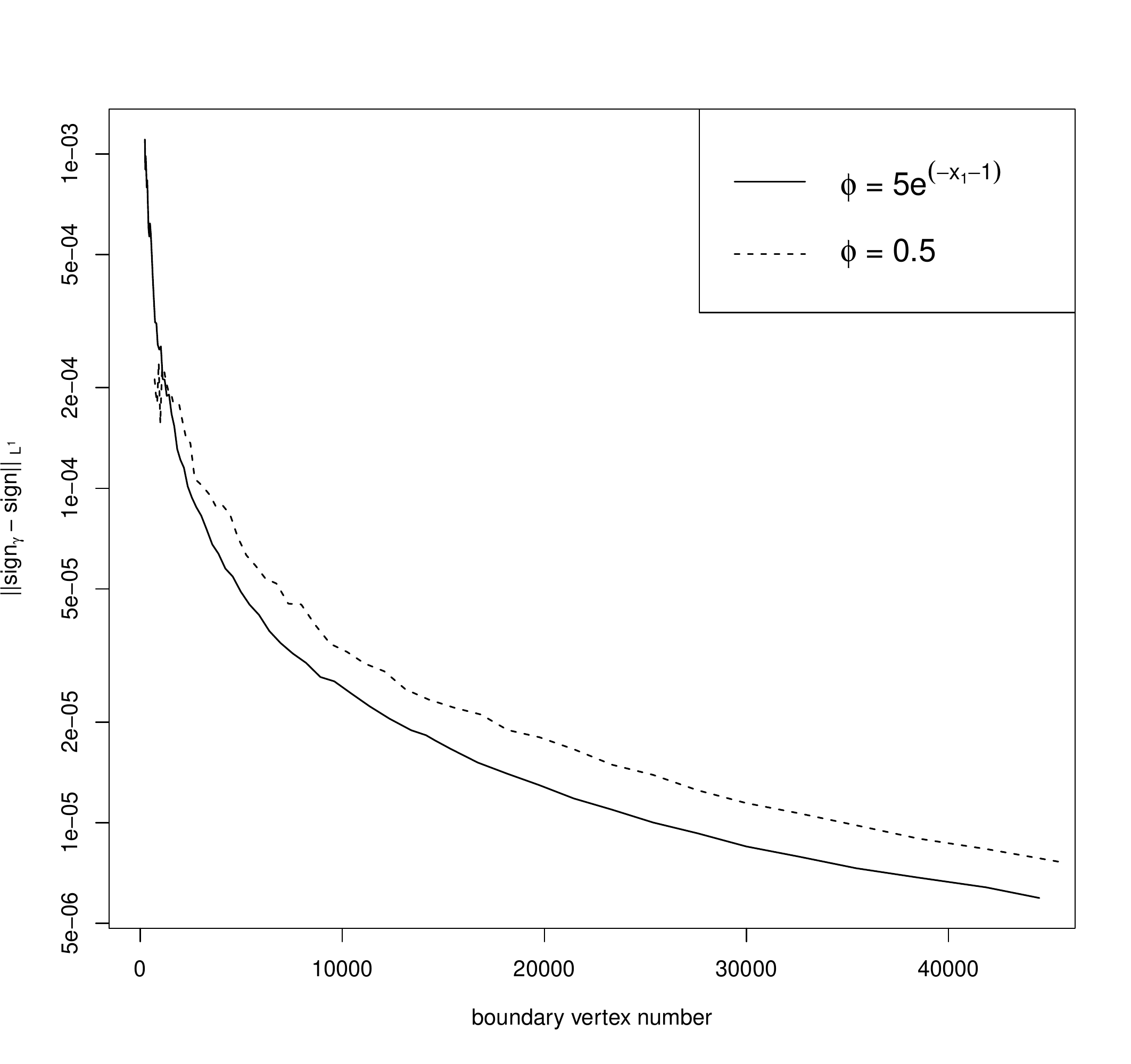}
		\caption{Graphs of $	\Vert\textup{sign}_\gamma\big(\bar{\lambda} + c\cdot(y_{\gamma, c} - \varphi)\big) - \textup{sign}\big(\bar{\lambda} + c\cdot(y_{c} - \varphi)\big)\Vert_{L^1(\Omega)}$ as functions the number of mesh vertices at the boundary of $A$ calculated on refined meshes with $c = 10^5$, $\gamma = 10^8$.}
		\label{SignDifferenceStudyPlot}
	\end{figure}
	
	The adjoints $p_{\gamma,c}$ and $p_c$ are calculated by solving \cref{SmoothedAdjoint} and \cref{adjoint_limitproblem} with first order elements by using the FEniCS standard linear algebra back end solver PETSc. 
	
	Calculating the limit $p$ of the adjoints $p_{\gamma,c}$ as in \cref{adjoint_unregu} and \cref{adjoint_subproblem} are performed in several steps. First, a linear system corresponding to 
	\begin{align}\label{adjointLinearSystemNum}
	\begin{split}
	-\Delta p  =& -(y-\bar{y}) \quad \text{in } \Omega \\
	p =&\; 0 \qquad \hspace{.9cm} \text{on } \partial \Omega 
	\end{split}
	\end{align}
	is assembled without incorporation of information from the active set $A$. Afterwards, the vertex indices corresponding to the points in the active set $A = \{x\in \Omega \;\vert\; y - \varphi \geq 0\}$ are collected by checking the condition 
	\begin{align}\label{ActiveSetNumRule}
	y(x) - \varphi(x) \geq -\varepsilon_{\text{adj}}
	\end{align}
	for some error bound $\varepsilon_{\text{adj}} >0$. The error bound $\varepsilon_{\text{adj}}$ is incorporated since $y$ is feasibly approximated by $y_i$ with the semi-smooth Newton method from \cite{ItoKunisch_VI}, i.e., $y_i \leq \varphi$ for all $i\in \mathbb{N}$. After this, the collected vertex indices are used to incorporate the Dirichlet boundary conditions $p = 0 \quad \text{in } A$ into the linear system corresponding to (\ref{adjointLinearSystemNum}).
	To solve the resulting system, we use the same procedures as to solve for $p_{\gamma,c}, p_c$, i.e.m the standard PETSc back end conjugate gradient solver. An exemplary solution $p$ of the unregularized adjoint equation is illustrated in \cref{AdjointPicFigure}. We want to point out that the active set and consequently the zero level set resulting from the Dirichlet conditions can be observed in \cref{AdjointPicFigure}.
	
	\begin{figure}
		\begin{subfigure}{.5\textwidth}
			\hspace{-1cm}	\includegraphics[scale=0.22]{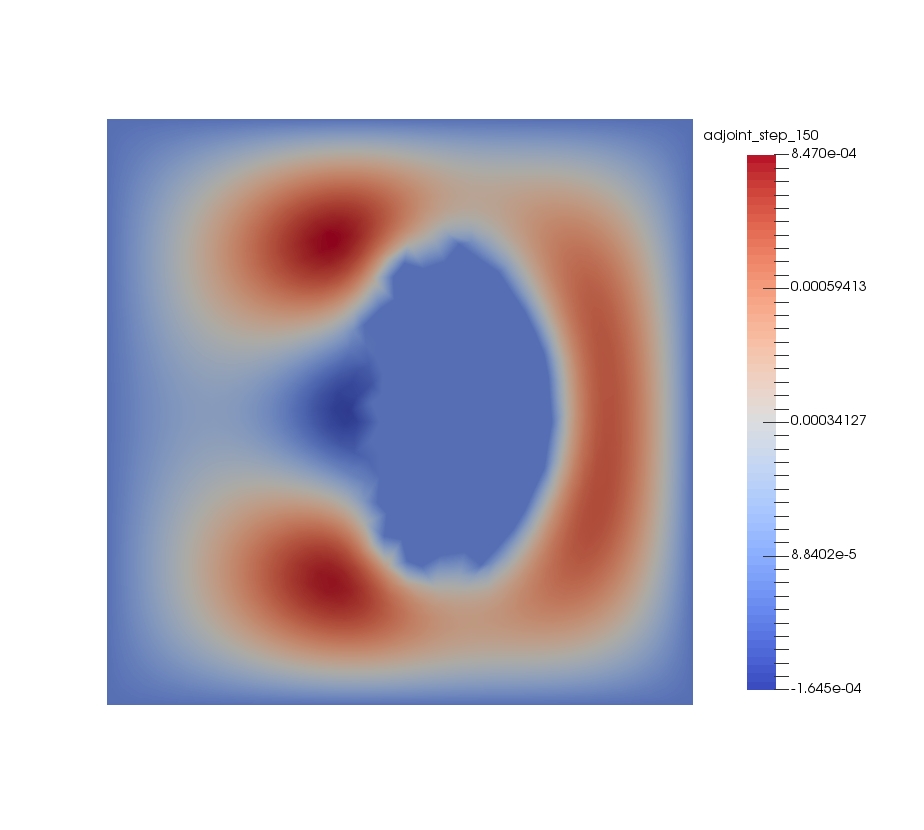}
		\end{subfigure}
		\begin{subfigure}{.45\textwidth}
			\hspace{-.5cm}\centering
			\vspace*{-.3cm}
			\includegraphics[scale=0.25]{adjoint_step_150_3d}
		\end{subfigure}
		\vspace{-.8cm}
		\caption{Solution of the adjoint $p$ in step $150$ for the unregularized equation (\ref{adjoint_unregu}) and $\varepsilon_{\text{adj}}=10^{-9}$. }
		\label{AdjointPicFigure}
	\end{figure}
	
	To calculate gradients $U\in H^1_0(\Omega, \mathbb{R}^2)$ used in the steepest descent method for solving \cref{NumProblem}, we use a Steklov-Poincar\'{e} metric induced by the linear elasticity equation, as proposed in \cite{schulz2015Steklov}. 
	In particular, we assemble the shape derivatives given in \cref{MainTheorem2} and \cref{MainTheorem3} as the right-hand side of the linear elasticity equation 
	\begin{align}\label{LinElas}
	\begin{split}
	\int_{\Omega}\sigma(U):\epsilon(V)\;dx &= DJ(\Omega)[V] \qquad \forall V\in H^1_0(\Omega, \mathbb{R}^2) \\
	\sigma(U):&= \lambda_{\text{elas}}\text{Tr}(U)I + 2\mu_{\text{elas}}\epsilon(U) \\
	\epsilon(U):&= \frac{1}{2}(\nabla U^T + \nabla U) \\
	\epsilon(V):&= \frac{1}{2}(\nabla V^T + \nabla V)
	\end{split}	
	\end{align}
	with the so called Lam\'{e}-parameters $\lambda_{\text{elas}}$ and $\mu_{\text{elas}}$. Here, we choose $\lambda_{\text{elas}} = 0$ and $\mu_{\text{elas}}$ as the solution of the Poisson problem
	\begin{equation}
	\begin{split}
	-\Delta \mu_{\text{elas}} &= 0 \qquad \;\;\; \text{in } \Omega \\
	\mu_{\text{elas}} &= \mu_{\text{max}} \quad \text{on } \Gamma \\
	\mu_{\text{elas}} &= \mu_{\text{min}} \quad \text{on } \partial\Omega
	\end{split}
	\end{equation}
	for $\mu_{\text{max}}, \mu_{\text{min}} > 0$. As a physical interpretation, this enables to control stiffness of the grid by choosing $\mu_{\text{max}}$ and $\mu_{\text{min}}$ in order to influence $\mu_{\text{elas}}$ acting as a coefficient function in the linear elasticity equation (\ref{LinElas}). Thus, larger values of $\mu_{\text{max}}$ lead to more stiffness at the interface $\Gamma$ and larger values of $\mu_{\text{min}}$ to more stiffness at the boundary $\partial\Omega$ of the hold-all domain $\Omega$. For our calculations, we choose $\mu_{\text{min}} = 0$ and $\mu_{\text{max}} = 25$ for $\varphi_1$ and $\mu_{\text{max}} = 55$ for $\varphi_2$. It is important to notice that we set all right-hand side values of (\ref{LinElas}) which do not have a neighboring vertex on the interface to $0$. For a more detailed discussion of this we refer to \cite{schulz2015Steklov}.
	
	\begin{figure}
		\centering
		\includegraphics[scale=1.]{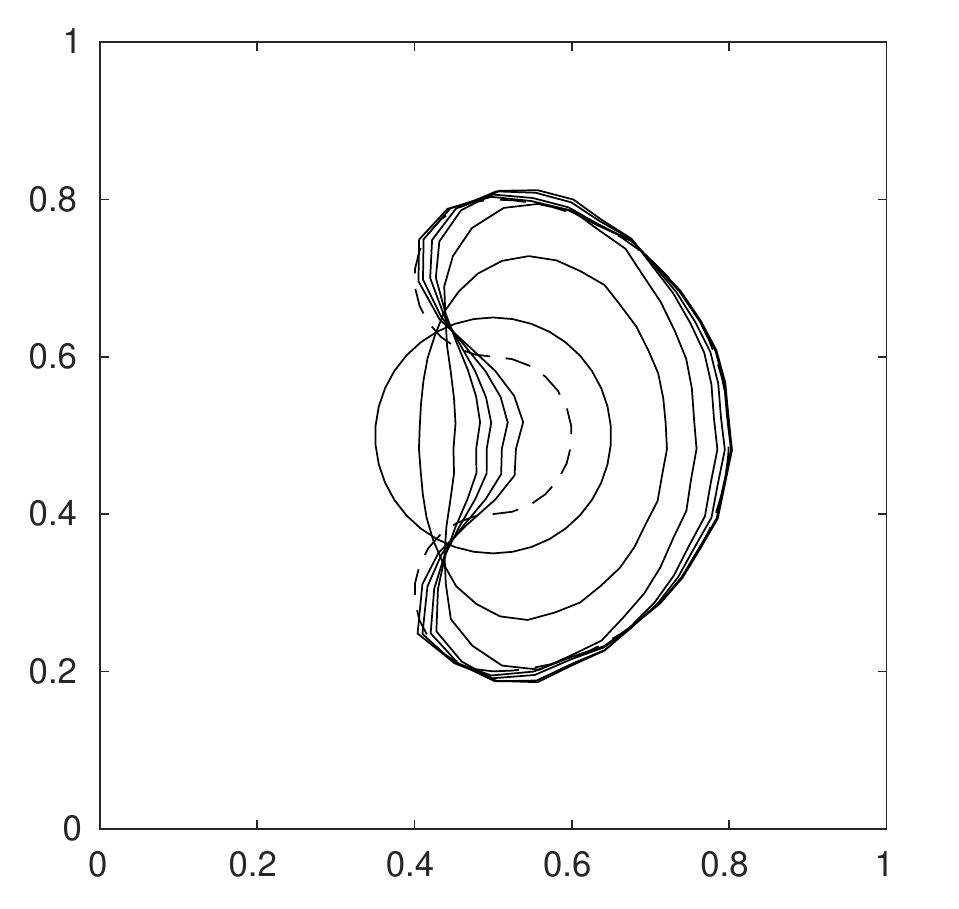}
		\caption{Interfaces of steps 0, 50, 150, 320, 450, 750, 1200 of the unregularized optimization procedure using $\varepsilon_{\text{adj}}=10^{-9}$ and $\varphi_2 = 5 e^{-x_1 -1}$. The target interface is represented with dotted lines, the start interface is the centered circle.}
		\label{ShapeMorphFigure}
	\end{figure}
	
	To complete the description of our optimization we shortly explain the linesearch we will employ in our numerical calculations. We use a simple backtracking linesearch with sufficient descent criterion, where $U_k$ denotes the shape derivative calculated at the corresponding interface in $\Omega_k$ in step number $k$, $\mathcal{T}_{\tilde{U}}(\Omega_k) := \{ y\in\mathbb{R}^2: y = x + \tilde{U}(x) \text{ for some } x\in\Omega_k\}$ the linearized vector transport by $\tilde{U}$ and $y_{\tilde{U}}$ the state solution in $\mathcal{T}_{\tilde{U}}(\Omega_k)$.
	
	\LinesNumbered
	\begin{algorithm2e}
		\caption{Backtracking linesearch.}
		\label{BacktrackingLinesearch}
		$\tilde{U} \gets U_k$ \\
		\While{$\mathcal{J}\big(y_{\tilde{U}}, \mathcal{T}_{\tilde{U}}(\Omega_k)\big) > 0.995\cdot \mathcal{J}\big(y_k, \Omega_k\big)$}{
			$\tilde{U} \gets 0.5\cdot\tilde{U}$
		}
		$\Omega_{k+1} \gets \mathcal{T}_{\tilde{U}}(\Omega_k)$
	\end{algorithm2e}
	
	We summarize our approach in algorithm~\ref{MainAlgo} for the unregularized procedures. The regularized and smoothed procedures work analogously by modifying the state, adjoint and shape derivative equations. The calculations of $p_{\gamma,c}, p_c$ are straightforward and need not the additional steps outlined in before and in algorithm \ref{MainAlgo} for the unregularized $p$. In the design of algorithm~\ref{MainAlgo} a safeguarding technique is employed. This stems from the fact, that the limit of shape derivatives $D\mathcal{J}$ from \cref{UnreguShapeDeriv} is in general not the true shape derivative of the initial problem, see \cref{RemarkNoShapeDer}. Hence an additional testing the convergence criterion for the fully regularized shape derivative $\mathcal{D}\mathcal{J}_{\gamma,c}$ is performed after the convergence by $D\mathcal{J}$. If no convergence is detected by $D\mathcal{J}_{\gamma,c}$, $D\mathcal{J}_{\gamma,c}$ will be used to calculate a further descent direction, as the latter is a true shape derivative by \cref{MainTheorem2}.  Further, the safeguarding acts as a safety when the adjoint limit object $p_k$ is flawed due to erroneous allocation of the active set $A_k$ as discussed with \cref{ActiveSetNumRule} in the beginning of this section. The smoothed model is not prone to this effect, hence acting as a substitute model for further gradient calculations.
	
	In our calculations the safeguard was never activated by not finding a descent direction during the linesearch procedure, indicating that the shape derivative limiting object $D\mathcal{J}$ is acting appropriately for a shape derivative substitute, making the safeguard for this purpose obsolete. Still, the safeguarding is activated at convergence for coarse grids or imprecise calculations of the state $y_k$, indicating a non-neglectable difference in $\Vert D\mathcal{J}_{\gamma,c}\Vert$ and $\Vert D\mathcal{J}\Vert$. This is only due to false allocation of the active set $A_k$, resulting in inaccurate $p_k$ and $\Vert D\mathcal{J}(\Omega_k)\Vert$. For grids with maximum cell diameter $0.01$ or less and  error tolerance $\varepsilon_{\text{state}} < 1.e-7$ for the state calculation, the errors in active set allocation are sufficiently small for $\Vert D\mathcal{J}(\Omega_k)\Vert \approx \Vert D\mathcal{J}_{\gamma,c}(\Omega_k)\Vert$ and the safeguarding to not be activated at convergence in all our examples.
	
	\DontPrintSemicolon
	\LinesNumbered
	
	\begin{algorithm2e}\caption{Shape optimization for unregularized VI constraints with safeguard strategy.}
		\label{MainAlgo}
		\SetAlgoRefName{MainAlgo}
		Set $\Omega_0, \varphi, f, \bar{\lambda}, \bar{y}, \gamma, c$ \; 
		\While{$\Vert D\mathcal{J}(\Omega_k)\Vert > \varepsilon_{\text{\emph{shape}}}$ \textbf{\textrm{or}} $\Vert D\mathcal{J}_{\gamma,c}(\Omega_k)\Vert > \varepsilon_{\text{\emph{shape}}}$}{
			Calculate state $y_k$ with tolerance  $\varepsilon_{\text{state}}$ \;
			Calculate 'adjoint' $p_k$ \;
			\Indp
			Assemble adjoint system (\ref{adjointLinearSystemNum}) neglecting active set \;
			Collect vertex indices of active set by (\ref{ActiveSetNumRule}) \;
			Implement Dirichlet conditions of active set \; 
			Solve modified adjoint linear system \;
			\Indm 
			Calculate $\Vert D\mathcal{J}(\Omega_k)\Vert$ and shape gradient $U_k$\;
			\Indp
			Assemble gradient system (\ref{LinElas})\;
			Set $D\mathcal{J}(\Omega_k)[V] = 0$ on all vertices without support at interface $\Gamma_{\text{int}}$\;
			Solve for gradient $U_k$\;
			\Indm
			Perform backtracking linesearch (algorithm \ref{BacktrackingLinesearch}) to get $\tilde{U}_k$\;
			\textbf{if} linesearch fails to give descent direction $\tilde{U}_k$ \textbf{\textrm{or}} $\Vert D\mathcal{J}(\Omega_k)\Vert \leq \varepsilon_{\text{\emph{shape}}}$ : \;
			\Indp
			Calculate fully regularized state $y_{\gamma,c}$ \;
			Calculate fully regularized adjoint $p_{\gamma,c}$ \;
			Calculate $\Vert D\mathcal{J}_{\gamma,c}(\Omega_k)\Vert$\;
			\textbf{if} $\Vert D\mathcal{J}_{\gamma,c}(\Omega_k)\Vert >\varepsilon_{\text{\emph{shape}}}$ \;
			\Indp
			Calculate fully regularized gradient $U_{\gamma,c}$ by $D\mathcal{J}_{\gamma,c}(\Omega_k)$ \;
			Perform backtracking linesearch (algorithm \ref{BacktrackingLinesearch}) to get $\tilde{U}_{k}$ \;
			\Indm
			\Indm
			$\Omega_{k+1} \gets \mathcal{T}_{\tilde{U}_k}(\Omega_k)$ 		
		}
	\end{algorithm2e}  
	
	
	\begin{figure}
		\includegraphics[scale=0.9]{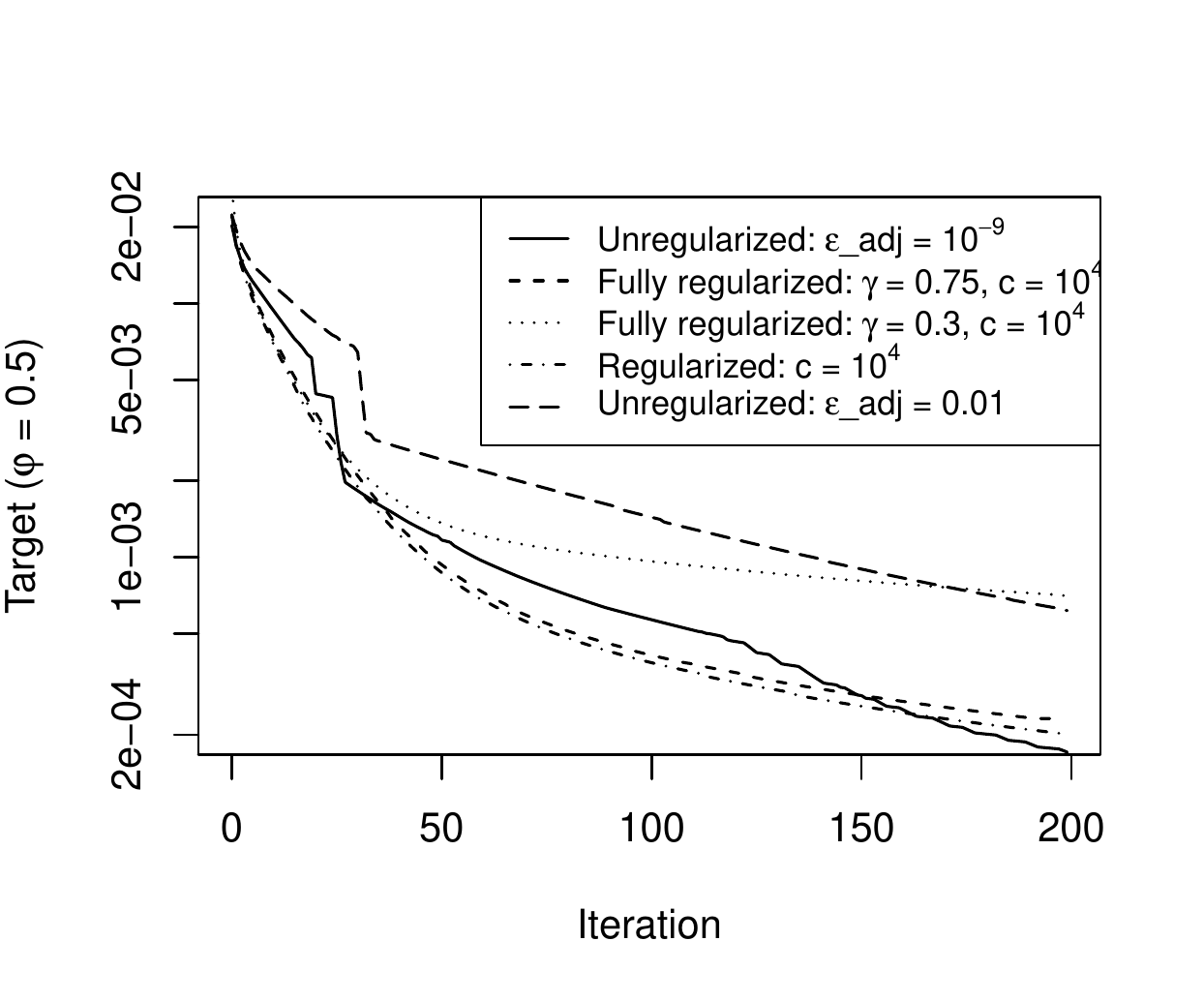}
		\vspace{-.5cm}
		\caption{Convergence plot of target functional $\mathcal{J}$ values for different regularization and unregularized approaches for obstacle $\varphi_1 = 0.5$ using steepest descent}
		\label{ConvPlotconstFigure}
	\end{figure}
	\begin{figure}
		\includegraphics[scale=0.9]{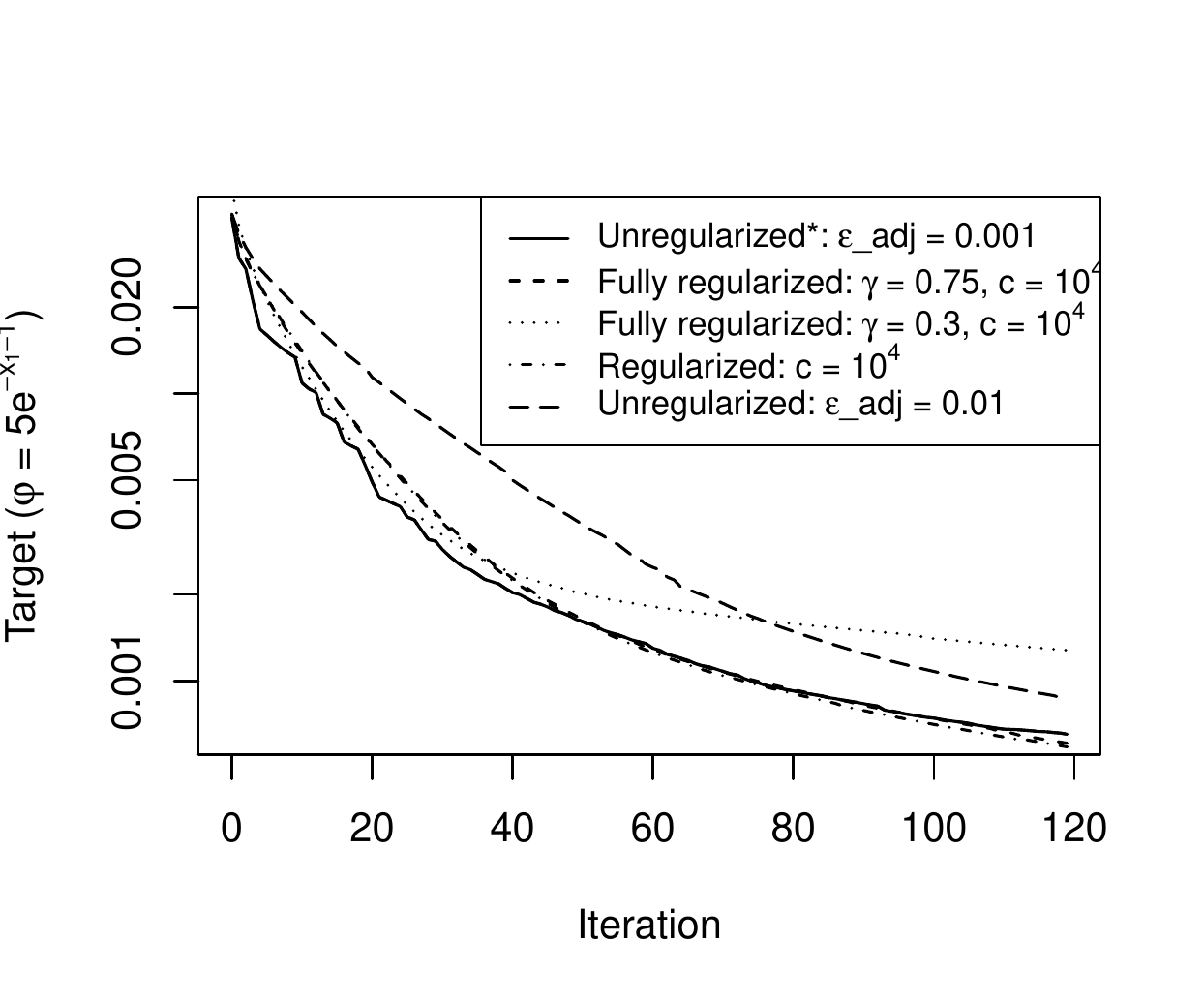}
		\vspace{-.5cm}
		\caption{Convergence plot of target functional $\mathcal{J}$ values for different regularization and unregularized approaches for obstacle $\varphi_2 = 5 e^{-x_1 -1}$. Unregularized* uses a lower tolerance $\varepsilon_{state}= 0.00001$ for the state calculation. Notice that regularized and fully regularized approaches for $\gamma = 0.75, c = 10^4$ are almost indistinguishable.}
		\label{ConvPlotexptFigure}
	\end{figure}
	
	Our findings concerning convergence of the various shape optimization approaches, using the unregularized approach for various $\varphi_{\text{adj}}$, as well as regularized approaches with different parameters $\gamma, c >0$, are displayed in \cref{ConvPlotconstFigure} for $\varphi_1 = 0.5$ and in \cref{ConvPlotexptFigure} for $\varphi_2 = 5 e^{-x_1 -1}$. Morphed shapes arising during the optimization procedure are plotted in \cref{ShapeMorphFigure} for the unregularized approach using $\varepsilon_{\text{adj}} = 10^{-9}$. It can be seen in the plots that there are vanishing difference between approaches using fully regularized calculation with sufficiently large $\gamma$ and $c$, regularized ones with large $c$ and the unregularized one. 
	For smaller regularization parameters $\gamma$ and $c$, the solved state and adjoint equations begin to differ from the original problem and, thus, slow down convergence, and for very small $\gamma$ and $c$ no convergence at all. 
	
	The convergence behavior of the unregularized method strongly depends on the selection of the active set. When the state solution $y$ is not calculated with sufficient precision the numerical errors lead to misclassification of vertex indices. Hence wrong Dirichlet conditions are incorporated in the adjoint system, creating errors in the adjoint. This makes the gradient sensitive to error for smaller $\varepsilon_{\text{adj}}$, as can be seen by the slight roughness of the target graphs in \cref{ConvPlotconstFigure} and \cref{ConvPlotexptFigure} for $\varepsilon_{\text{adj}} = 10^{-9}$ and $\varepsilon_{\text{adj}} = 0.01$. 
	In order to compensate this, the condition for checking active set indices (\ref{ActiveSetNumRule}) can be relaxed by increasing $\varepsilon_{\text{adj}}$. This increases likelihood of correctly classifying the true active indices, while also increasing likelihood of misclassification of inactive indicies. Such a relaxation can lead to errors in the adjoint increasing with $\varepsilon_{\text{adj}}$ and, thus, trading convergence speed for robustness, also visible in \cref{ConvPlotconstFigure} and \cref{ConvPlotexptFigure}. Of course, this gets less feasible for highly oscillatory obstacle $\varphi$ and state $y$, as well as state solves with high tolerance $\varepsilon_{\text{\text{state}}}$.
	
	In order to circumvent this, it is obviously sufficient to decrease error tolerance $\varepsilon_{\text{state}}$ of the state calculation. An exemplary result of this can be seen in \cref{ConvPlotexptFigure} under unregularized*, where we decreased the error tolerance to $\varepsilon_{\text{state}} = 4.e-5$. Nevertheless, additional decrease of $\varepsilon_{\text{state}}$ comes with more computational cost, whereas with increase of $\varepsilon_{\text{adj}}$ the robustness is paid by loss of convergence speed.
	
	It is worth to mention that implementing the unregularized state and adjoint becomes especially numerically exploitable with higher resolution meshes and more strongly binding obstacles $\varphi$, i.e., larger active sets $A$. This is possible by sparse solvers due to the incorporation of Dirichlet conditions on the active set, as we have proposed, or by a fat boundary method as in \cite{maury2001fat}. This is especially advantageous for large systems resulting from fine resolution grids, as exploitability of sparsity and accuracy of our method increase at the same time.
	
	So in contrast to the method proposed in \cite{mPDAS}, where performance slows down for more active obstacle $\varphi$, we do not notice unusual slowdown in performance with the methods proposed in this article, and even offer possibility to actually benefit numerically from more binding obstacle $\varphi$.

	\section{Conclusion}
	Shape optimization for  variational inequalities is more challenging than both, elliptic shape optimization and optimal control for variational inequalities. In this paper, we derive optimality conditions for shape optimization in the context of variational inequalities in the flavor of optimal control problems. Regularized variants are studied and limiting conditions derived. This gives rise to highly efficient optimization algorithms. In the future general investigations of necessary optimality criteria for VI constrained shape optimization like C-stationarity are conceivable. Also large-scale multidimensional computational comparisons of the presented method in comparison to other state-of-the-art methods is of particular interest.
	
	
	\section*{Acknowledgement}
	
	The authors are indebted to Leonhard Frerick (Trier University) for many helpful comments and discussions about functional analytical aspects of convergence.
	Furthermore, the paper profited form remarks of Gerd Wachsmuth on an earlier version.
	This work has been partly supported by the German Research Foundation (DFG) within the priority program ``Non-smooth and \discretionary{Complementarity-}{based}{Complementarity-based} Distributed Parameter Systems: Simulation and Hierarchical Optimization'' 
	SPP 1962/1 and SPP 1962/2 under contract numbers Schu804/15-1 and WE 6629/1-1, by the DFG research training group 2126 Algorithmic Optimization, and by the BMBF (Bundesministerium f\"{u}r Bildung und Forschung) within the collaborative project GIVEN (FKZ: 05M18UTA).

	
	\newpage
	\appendix
	\section{Proof of existence of shape derivatives in \cref{MainTheorem2}}
	\label{Proof_ExistenceShapeDeriv}
	What follows is a proof for the existence of shape derivatives $D\mathcal{J}_{\gamma,c}$ for all $\gamma,c >0$ under the assumptions of \cref{MainTheorem2}.
	\begin{proof}
		
		For the proof of existence $D\mathcal{J}_{\gamma,c}$ for all $\gamma,c >0$, we will employ the so called averaged adjoint approach, as found in \cite[Ch. 7, Thm. 7.1]{sturm}. We roughly follow a proof found in \cite[Ch. 7, Thm. 7.2]{sturm}, but have to give a proof ourselves, since the situation in this paper differs from the one mentioned previously, e.g. as we are not having a bounded Nemetskii operator as the non-linearity in the semi-linear state equation.
		
		Let $\gamma,c >0$. By definition, the Lagrangian function corresponding to our problem is given by
		\begin{align}\label{LagrangianGammaC}
		\begin{split}
		\mathcal{L}_{\gamma, c}(\Omega, y_{\gamma, c}, p_{\gamma, c}) = &\frac{1}{2}\int_{\Omega} (y_{\gamma, c} -\bar{y})^2dx  + a(y_{\gamma,c}, p_{\gamma,c}) \\
		&+ ( \text{max}_\gamma\big(\bar{\lambda} + c\cdot(y_{\gamma, c} - \varphi)\big), p_{\gamma, c})_{L^2(\Omega)} - (f,p_{\gamma, c})_{L^2(\Omega)}
		\end{split}
		\end{align}
		
		We have to proof the assumptions (H0) - (H3) needed to apply the averaged adjoint theorem as found in \cite[Chapter 7, Thm. 7.1]{sturm} in order to guarantee shape differentiability. For convenience, we do not state all the lengthy assumptions here, and refer the interested reader to \cite[Chapter 7, Thm. 7.1]{sturm}. Let us fix a deformation vector field $V$ and denote the domains deformed by $\Omega_t$ for deformation parameters $t\in[0,\tau]$ and some $\tau >0$ small enough, such that the corresponding deformations $\mathcal{T}_{t}$ is bijective. For the rest of the existence proof, we will drop $\gamma,c$ as the subscripts of $y_{\gamma,c}$ and $p_{\gamma,c}$ for readability purposes, still knowing we are in the fully regularized situation. Define
		\begin{align}\label{LagrangianGammaC_inV}
		G: [0,\tau] \times H^1_0(\Omega) \times H^1_0(\Omega) \rightarrow \R, (t,y,p) \mapsto \mathcal{L}(\Omega_t, y\circ \mathcal{T}_t^{-1}, p\circ \mathcal{T}_t^{-1})
		\end{align}	
		for the deformed domain $\Omega_t$ resulting from application of the deformation  $\mathcal{T}_t$ in direction $V$ parametrized by $t\in [0,\tau]$.
		
		The first assumption (H0) concerns well behavedness of \cref{LagrangianGammaC_inV}. First notice that the function in \cref{LagrangianGammaC_inV} is both differentiable in the state $y$ and adjoint $p$, which can be seen after applying the transformation theorem to \cref{LagrangianGammaC_inV}.
		Thus for the set	
		\begin{align}
		X(t) := \{ \hat{y} \in H^1_0(\Omega)\; \vert \; \underset{y \in H^1_0(\Omega)}{\inf}\underset{p \in H^1_0(\Omega)}{\sup} G(t, y, p)\; = \underset{p \in H^1_0(\Omega)}{\sup} G(t, \hat{y}, p) \; \}
		\end{align}	
		we have 
		\begin{equation}
		\label{proofExistenceShapederiv0}
		X(t) = \{y^t\} \subset{H^1_0(\Omega)} \quad \forall t\in[0,\tau],
		\end{equation}	
		with $y^t = y_t\circ\mathcal{T}_t$ being the retraction of the unique solution $y_t \subset H^1_0(\Omega_t)$ of the fully regularized state equation \cref{SmoothedState} on the deformed domain $\Omega_t$ as by the use of Minty-Browder's Theorem as portrayed in the proof of \cref{Theorem0}. 
		
		Further, the Lagrangian in direction $V$ as a function of averaged states inserted 	
		\begin{equation}
		[0,1] \rightarrow \R, \; s \mapsto G(t, sy^t + (1-s)y^0, p)
		\end{equation}
		is absolutely continuous in $s\in [0,1]$. For this, we make use of the fact that
		\begin{equation}\label{EmbeddingLp}
		L^q(\Omega) \hookrightarrow L^p(\Omega)\quad \text{for }1\leq p < q \leq \infty
		\end{equation}  
		for bounded $\Omega\subset \R^n$ with a constant depending on $p$ and $q$. So absolute continuity is satisfied, as we have existing derivatives of the integrand due to \cref{AssumptionsOnMax} (i) and integrability of the integrand due to
		\begin{align}\label{max_gammaIntegrability}
		\begin{split}
		&\Vert \text{max}_{\gamma}(\bar{\lambda} + c\cdot(z - \varphi)) \Vert_{L^2(\Omega)} \\
		\leq& \Vert \text{max}_{\gamma}(\bar{\lambda} + c\cdot(z - \varphi)) - \max(0,\bar{\lambda} + c\cdot(z - \varphi)) \Vert_{L^2(\Omega)} \\
		\quad&\;+ \Vert \max(0, \bar{\lambda} + c\cdot(z - \varphi)) \Vert_{L^2(\Omega)} \\
		\leq& \text{vol}(\Omega)^{\frac{1}{2}}g(\gamma) + \Vert \bar{\lambda} + c\cdot(z - \varphi) \Vert_{L^2(\Omega)} \\
		\leq& \text{vol}(\Omega)^{\frac{1}{2}}g(\gamma) + \Vert \bar{\lambda} + c\cdot z \Vert_{L^2(\Omega)} +  C\cdot c \cdot \Vert \varphi \Vert_{H^1_0(\Omega)} \\
		\leq&  \text{vol}(\Omega)^{\frac{1}{2}}g(\gamma) + \Vert \bar{\lambda}  \Vert_{L^2(\Omega)} +  C\cdot c \cdot( \Vert \varphi \Vert_{H^1_0(\Omega)} +  \Vert z \Vert_{H^1_0(\Omega)} )< \infty
		\end{split}
		\end{align}
		for all $z \in L^2(\Omega)$ and all $z \in H^1_0(\Omega)$ as by \cref{AssumptionsOnMax} (ii), $\Omega$ being bounded and $H^1_0(\Omega) \hookrightarrow L^4(\Omega) \hookrightarrow L^2(\Omega)$ with constant $C$ by \cref{EmbeddingH1inL4} and \cref{EmbeddingLp}. Further, the directional derivative mapping	
		\begin{equation}
		[0,1] \rightarrow \R, \; s \mapsto \frac{\partial}{\partial y} G(t, sy^t + (1-s)y^0, p; \tilde{p})
		\end{equation}	
		is integrable for all $\tilde{p} \in H^1_0(\Omega)$, since	
		\begin{align*}
		&\Vert \frac{\partial}{\partial y}G(t, sy^t + (1-s)y^0, p; \tilde{p}) \Vert_{L^1(0,1)} \\
		=& \int_{0}^{1} \vert a_t(p, \tilde{p}) + c\cdot\Big(\text{sign}_\gamma\big(\bar{\lambda} + c\cdot(sy^t + (1-s)y^0 - \varphi)\big)p, \tilde{p}\Big)_{L^2(\Omega)} \\ 
		&\qquad+ ( sy^t + (1-s)y^0 - \bar{y}, \tilde{p})_{L^2(\Omega)} \vert \;ds \\
		\leq& \int_{0}^{1} B_t \cdot \Vert p \Vert_{H^1_0(\Omega)} \Vert \tilde{p}\Vert_{H^1_0(\Omega)} + c \cdot \Vert p \Vert_{H^1_0(\Omega)} \Vert \tilde{p}\Vert_{H^1_0(\Omega)} \\
		&\qquad + (s \Vert y^t \Vert_{H^1_0(\Omega)} + (1-s) \Vert y^0 \Vert_{H^1_0(\Omega)} + \Vert \bar{y} \Vert_{L^2(\Omega)}) \Vert \tilde{p}\Vert_{H^1_0(\Omega)} \;ds \\
		=& (B_t + c)\Vert p \Vert_{H^1_0(\Omega)} \Vert \tilde{p}\Vert_{H^1_0(\Omega)} + (\frac{1}{2}\Vert y^t \Vert_{H^1_0(\Omega)} \\
		&+ \frac{1}{2}\Vert y^0 \Vert_{H^1_0(\Omega)} + \Vert \bar{y} \Vert_{L^2(\Omega)})\Vert \tilde{p}\Vert_{H^1_0(\Omega)} < \infty
		\end{align*}	
		by H\"older's Inequality, \cref{AssumptionsOnMax} (iii) and $a_t(.,.)$ being the  bilinear form defined by retraction of $a(.,.)$ from $\Omega_t$ to $\Omega$ bound with constants $B_t>0$. This, and the easy to verify fact that $G$ is a affine linear function in $p$, gives us (H0). We remind the careful reader, that the Jacobians created by retraction of $\Omega_t$ to $\Omega$ are to be implicitly included in scalarproducts and norms above for the calculations to be valid. We don't explicitly state these for readability.
		
		Next we introduce the set of so called averaged adjoints	
		\begin{equation}\label{AveragedAdjSet}
		Y(t, y^t, y^0) := \Big\{ q \in H^1_0(\Omega) \; \vert \; \int_{0}^{1} \frac{\partial}{\partial y} G(t, sy^t + (1-s)y^0, q; \tilde{p}) \;ds = 0\;\; \forall \tilde{p}\in H^1_0(\Omega) \Big\}
		\end{equation}
		We manipulate the averaged adjoint equation found in \cref{AveragedAdjSet} by interchanging integrals	
		\begin{align}\label{AveragedAdjEq}
		\begin{split}
		0 = &\int_{0}^{1} \frac{\partial}{\partial y} G(t, sy^t + (1-s)y^0, q; \tilde{p}) \;ds \\
		=& \int_{0}^{1}  a_t(q, \tilde{p}) + c\cdot\Big(\text{sign}_\gamma\big(\bar{\lambda} + c\cdot(sy^t + (1-s)y^0 - \varphi)\big)q, \tilde{p}\Big)_{L^2(\Omega)} \\
		&\qquad+ \big( sy^t + (1-s)y^0 - \bar{y}, \tilde{p}\big)_{L^2(\Omega)} \;ds \\
		=&\; a_t(q, \tilde{p}) + c\cdot \Big( (\int_{0}^{1}\text{sign}_{\gamma}(\bar{\lambda} + c\cdot( sy^t + (1-s)y^0 - \varphi  ))ds)q, \tilde{p}\Big)_{L^2(\Omega)} \\
		&\qquad + (\frac{1}{2}y^t + \frac{1}{2}y^0 - \bar{y}, \tilde{p})_{L^2(\Omega)}
		\end{split}
		\end{align}	
		which is an elliptic PDE with an additional positive $L^\infty(\Omega)$ coefficient function term for the zero'th order terms, where we again omitted explicit statement of Jacobians. By \cref{AssumptionsOnMax} $(iii)$, the additional coefficient term for the zero'th order terms in the averaged adjoint equation \cref{AveragedAdjEq} satisfies	
		\begin{align}\label{AveragedAdjointZeroOrderCoef}
		0 \leq \int_{0}^{1}\text{sign}_{\gamma}(\bar{\lambda}_t + c\cdot( sy^t + (1-s)y^0 - \varphi_t  ))ds \leq 1 \qquad \forall t\in [0,\tau]
		\end{align}
		which results in coercivity and boundedness of the corresponding bilinear form of the averaged adjoint equation. This lets us apply the Lemma of Lax-Milgram, resulting in existence of a unique solution for the averaged adjoint equation we will denote by $q^t\in H^1_0(\Omega)$ for all $t \in [0,\tau]$. Thus we have the identity $Y(t, y^t, y^0) = \{ q^t \} \subset H^1_0(\Omega)$, which together with \cref{proofExistenceShapederiv0} ensures condition (H2).
		
		We also notice, that the derivatives of $\frac{\partial}{\partial t}G$ exist and can be explicitly calculated after application of the transformation theorem, giving us (H1).
		
		To apply the averaged adjoint theorem from \cite{sturm2015shape}, it remains to address (H3), which is satisfied in our case by application of \cite[Lemma 4.1]{DissSturm}, if for the unique solutions of the state- and adjoint equation $y^0\in X(0)$ and $q^0\in Y(0, y^0, y^0)$ and a given sequence $(t_n)_{n\in\mathbb{N}}\subset [0,\tau]$ converging to zero, we can find a  subsequence $(t_{n_k})_{k\in\mathbb{N}}\subseteq(t_n)_{n\in\mathbb{N}}$ with $q^{t_{n_k}} \in Y(t_{n_k}, y^t_{n_k}, y^0)$ such that
		\begin{equation}\label{AveragedAdjointH3}
		\underset{\substack{k \rightarrow \infty \\ t\searrow0}}{\lim} \frac{\partial}{\partial t}G(t, y^0, q^{t_{n_k}}) = \frac{\partial}{\partial t}G(0, y^0, q^0)
		\end{equation}
		We will mimic parts of the argumentation found in the proof of \cite[Theorem 5.1]{DissSturm} accustomed to our situation, which is slightly different than the one found in \cite[Theorem 5.1]{DissSturm} or \cite[Theorem 7.2]{sturm}.
		
		Consider the solutions $y^0\in X(0)$ and $q^0\in Y(0, y^0, y^0)$ and a sequence \newline $(t_n)_{n\in\mathbb{N}}\subset [0,\tau]$ converging to zero.
		
		First notice that by monotony of the Nemetskii operator \cref{NymetskiyGamma} of the concerning semilinear state equation we have
		\begin{equation}
		(\text{max}_{\gamma}(\bar{\lambda} + c\cdot(z - \varphi)), z)_{L^2(\Omega)} \geq (\text{max}_{\gamma}(\bar{\lambda} - c\cdot\varphi), z)_{L^2(\Omega)} \qquad \forall z\in L^2(\Omega)
		\end{equation}
		This in turn, together with the coercivity of the retracted bilinearform $a_t(.,.)$ with constant $K_t >0$, \cref{max_gammaIntegrability} and choosing $y^t\in X(t) \subset H^1_0(\Omega)$ as a testfunction gives us	
		\begin{align}
		\begin{split}
		0 \leq \Vert y^t \Vert_{H^1_0(\Omega)}^2 \leq& K_t \cdot a_t(y^t, y^t) \\
		=& K_t \int_{\Omega}f \cdot y^t - \text{max}_\gamma(\bar{\lambda} + c\cdot(y^t - \varphi)) \cdot y^t \;dx \\
		\leq& K_t \int_{\Omega}f \cdot y^t - \text{max}_\gamma(\bar{\lambda} - c \cdot \varphi)) \cdot y^t \;dx \\
		\leq& K_t\cdot(\Vert f \Vert_{L^2(\Omega)} + \Vert\text{max}_\gamma(\bar{\lambda} - c \cdot \varphi))\Vert_{L^2(\Omega)})\cdot \Vert y^t \Vert_{H^1_0(\Omega)} < \infty
		\end{split}
		\end{align}
		again omitting Jacobians. Dividing by $\Vert y^t \Vert_{H^1_0(\Omega)}$, using the convergence $\mathcal{T}_t$ to the identity for $t\downarrow0$ and by taking a supremum we achieve 		
		\begin{align}\label{BoundednessRetractedStates}
		\begin{split}
		\Vert y^t \Vert_{H^1_0(\Omega)} \leq \underset{t\in \{t_n\}}{\sup}\Big(K_t\cdot(\Vert f \Vert_{L^2(\Omega)} + \Vert\text{max}_\gamma(\bar{\lambda} - c \cdot \varphi))\Vert_{L^2(\Omega)})\Big) =: M < \infty,
		\end{split}
		\end{align}	
		bounding the norms by a constant $0 < M < \infty$ independent of $t\in[0,\tau]$. Recognize that the norms still implicitly depend on $t$, since Jacobians are to be included.
		
		In the same line of argumentation we can confirm the boundedness of $\Vert q^t \Vert_{H^1_0(\Omega)}$. For this, we apply the first inequality of  \cref{AveragedAdjointZeroOrderCoef} to get
		\begin{align*}
		0 \leq&\; \Vert q^t \Vert_{H^1_0(\Omega)}^2 \leq K_t \cdot a_t(q^t, q^t) \\
		=& -K_t \Bigg(c\cdot \Big( \big(\int_{0}^{1}\text{sign}_{\gamma}(\bar{\lambda} + c\cdot( sy^t + (1-s)y^0 - \varphi ))ds\big)\cdot q^t, q^t\Big)_{L^2(\Omega)}
		\\
		&\qquad \qquad + (\frac{1}{2}y^t + \frac{1}{2}y^0 - \bar{y}, q^t)_{L^2(\Omega)}\Bigg) \\
		\leq& K_t (\Vert y^t \Vert_{H^1_0(\Omega)} + \Vert y^0 \Vert_{H^1_0(\Omega)} + \Vert \bar{y} \Vert_{L^2(\Omega)})\Vert q^t \Vert_{H^1_0(\Omega)} 
		\end{align*}
		By finally using \cref{BoundednessRetractedStates} we arrive at
		\begin{equation}\label{BoundednessRetractedAdjointsAveraged}
		\Vert q^t \Vert_{H^1_0(\Omega)} \leq \underset{t\in \{t_n\}}{\sup}\Big(K_t\cdot ( M + \Vert y^0 \Vert_{H^1_0(\Omega)} + \Vert \bar{y} \Vert_{L^2(\Omega)})\Big) < \infty
		\end{equation}	
		As we have established bound \cref{BoundednessRetractedStates}, we can choose a subsequence $(t_{n_k})_{k\in\mathbb{N}}\subseteq(t_n)_{n\in\mathbb{N}}$, such that $y^{t_{n_k}} \rightharpoonup z$ weakly in $H^1_0(\Omega)$ for $k \rightarrow \infty$ and some $z \in H^1_0(\Omega)$. 
		
		Further, using the convergence of the retracted functions $\bar{\lambda}_t$ and  $\varphi_t$ in $L^2(\Omega)$ for $t\downarrow0$, we can uniformly and independently of $t_{n_k}$ bound	
		\begin{align*}
		&\Vert \text{max}_\gamma( \bar{\lambda} + c\cdot(y^{t_{n_k}} - \varphi))\Vert_{L^2(\Omega)} \\
		\leq&\; \text{vol}(\Omega)^{\frac{1}{2}}g(\gamma) + c\cdot M + \underset{t \in \{t_{n_k}\}}{\sup} 
		\Vert \bar{\lambda} - c \cdot \varphi \Vert_{L^2(\Omega)} < \infty
		\end{align*}	
		by using \cref{max_gammaIntegrability} and \cref{BoundednessRetractedStates}. By having boundedness of the coercive bilinearforms $a_t(.,.)$, \cref{BoundednessRetractedStates} and smoothness in $\text{max}_{\gamma}$ by \cref{AssumptionsOnMax} (i) we are able to apply Lebesgue's dominated convergence theorem to the retracted state equations, giving us $y_{t_{n_k}} \rightharpoonup y^0$ weakly in $H^1_0(\Omega)$ due to the unique solution guaranteed by Minty-Browder's theorem.
		
		Applying the same routine due to \cref{BoundednessRetractedAdjointsAveraged}, we can choose a subsequence of $\{t_{n_k}\}$, which we will again call $\{t_{n_k}\}$ by abuse of notation, such that $q^{t_{n_k}} \rightharpoonup u$ weakly in $H^1_0(\Omega)$ for some $u \in H^1_0(\Omega)$. Then uniform boundedness \cref{AveragedAdjointZeroOrderCoef}, the previously established weak convergence $y_{t_{n_k}} \rightharpoonup y^0$ and \cref{BoundednessRetractedAdjointsAveraged} yield applicability of Lebegue's theorem for inserted $t_{n_k}$ in \cref{AveragedAdjEq}. For $k\rightarrow \infty$, the limit equation of \cref{AveragedAdjEq} is the fully regularized adjoint equation \cref{SmoothedAdjoint}, which has a unique solution by Lax-Milgram's lemma. Whence $q^{t_{n_k}} \rightharpoonup q^0 = p_{\gamma,c} \in Y(0, y^0, y^0)$ weakly in $H^1_0(\Omega)$ with the previously established weak convergence of $q^{t_{n_k}}$ and continuity of $\text{sign}_\gamma$ by \cref{AssumptionsOnMax} (i).
		
		Now we have found a subsequence $\{t_{n_k}\}\subseteq \{t_n\}$, such that $q^{t_{n_k}} \rightharpoonup q^0$ weakly in $H^1_0(\Omega)$. Using the transformation theorem, $G(t, y^0,q^{t_{n_k}})$ from \cref{LagrangianGammaC_inV} can be stated as an integral in $\Omega$ with integrands being differentiable in $t\in[0,\tau]$. The derivative $\frac{\partial}{\partial t}G(t, y^0,q^{t_{n_k}})$ is weakly continuous in it's first and last argument, hence the weak convergence $q^{t_{n_k}} \rightharpoonup q^0$ implies \cref{AveragedAdjointH3}, which is condition (H3) from \cite[Theorem 7.1]{sturm}. All assumptions (H0)-(H3) for the averaged adjoint theorem \cite[Theorem 7.1]{sturm} are satisfied, finally guaranteeing existence of shape derivatives $D\mathcal{J}_{\gamma,c}$ for all $\gamma,c >0$.
	\end{proof}

	\newpage
	\bibliographystyle{plain}
	\bibliography{citations.bib}

\end{document}